\documentclass[final,leqno]{siamltex}

\usepackage{mathtools}
\usepackage{amssymb}
\usepackage{stmaryrd}
\usepackage{color}
\usepackage{pdflscape}
\usepackage{epsfig}

\usepackage{epstopdf}
\usepackage{multirow}
\usepackage{hhline}
\usepackage{txfonts}
\usepackage{mathrsfs}
\usepackage{color}

\usepackage{subcaption}

\usepackage{mathrsfs}
\newtheorem{thm}{Theorem}[section]

\newtheorem{lem}[thm]{Lemma}
\newtheorem{rem}[thm]{Remark}

\numberwithin{equation}{section}

\newcommand{\norm}[1]{\left\Vert#1\right\Vert}
\newcommand{\abs}[1]{\left\vert#1\right\vert}

\title{Petrov-Galerkin and Spectral Collocation Methods for distributed Order Differential Equations} 

\author{
Ehsan Kharazmi
\footnote{D\lowercase{epartment of} C\lowercase{omputational} M\lowercase{athematics}, S\lowercase{cience}, \lowercase{and}, E\lowercase{ngineering} \& D\lowercase{epartment of} M\lowercase{echanical} E\lowercase{ngineering},	
M\lowercase{ichigan} S\lowercase{tate} U\lowercase{niversity}, 428 S S\lowercase{haw} L\lowercase{ane}, E\lowercase{ast} L\lowercase{ansing}, MI 48824, USA}
, Mohsen Zayernouri
\footnote{D\lowercase{epartment of} C\lowercase{omputational} M\lowercase{athematics}, S\lowercase{cience}, \lowercase{and}, E\lowercase{ngineering} \&
D\lowercase{epartment of} M\lowercase{echanical} E\lowercase{ngineering},	
M\lowercase{ichigan} S\lowercase{tate} U\lowercase{niversity}, 428 S S\lowercase{haw} L\lowercase{ane}, E\lowercase{ast} L\lowercase{ansing}, MI 48824, USA,  C\lowercase{orresponding author; zayern@msu.edu}}
 AND George Em Karniadakis
\footnote{D\lowercase{ivision of} A\lowercase{pplied} M\lowercase{athematics}, B\lowercase{rown} U\lowercase{niversity}, 182 G\lowercase{eorge} S\lowercase{treet}, P\lowercase{rovidence}, RI 02912, USA}
}

\begin{document}
\maketitle
\begin{abstract}
Distributed order fractional operators offer a rigorous tool for mathematical modelling of multi-physics phenomena, where the differential orders are distributed over a range of values rather than being just a fixed integer/fraction as it is in standard/fractional ODEs/PDEs. We develop two spectrally-accurate schemes, namely a Petrov-Galerkin spectral method and a spectral collocation method for distributed order fractional differential equations. These schemes are developed based on the fractional Sturm-Liouville eigen-problems (FSLPs) \cite{Zayernouri2013} . In the Petrov-Galerkin method, we employ fractional (non-polynomial) basis functions, called \textit{Jacobi poly-fractonomials}, which are the eigenfunctions of the FSLP of first kind, while, we employ another space of test functions as the span of poly-fractonomial eigenfunctions of the FSLP of second kind. We define the underlying \textit{distributed Sobolev space} and the associated norms, where we carry out the corresponding discrete stability and error analyses of the proposed scheme. In the collocation scheme, we employ fractional (non-polynomial) Lagrange interpolants satisfying the Kronecker delta property at the collocation points. Subsequently, we obtain the corresponding distributed differentiation matrices to be employed in the discretization of the strong problem. We perform systematic numerical tests to demonstrate the efficiency and conditioning of each method.

\end{abstract}
\pagestyle{myheadings}
\thispagestyle{plain}

%
\section{Introduction}
\label{Sec: Introduction}
%
Fractional differential equations (FDEs) seamlessly generalize the notion of standard (integer-order) differential equations to those of fractional order \cite{samkofractional1993,Kilbass2006,Podlubny99}. There exists plenty experimental evidence revealing the anomalous transport and nonlocal history dependent effect in complex physical systems. Fractional calculus and FDEs open up new possibilities for robust modeling of such complex multi-scale problems. Examples include: non-Gaussian (L\'{e}vy flights) processes in turbulent flows \cite{sreenivasan1997phenomenology, jha2003evidence, Castillo2004Plasma}, non-Newtonian fluids and rheology \cite{jaishankar2013,naghibolhosseini2015estimation}, non-Brownian transport phenomena in porous and disordered materials \cite{meer01, meral2010fractional}, and non-Markovian processes in multi-scale complex fluids and multi-phase applications \cite{jaishankar2014}. Over the past two decades, an extensive amount of work has been done developing numerical schemes for FDEs such as variational iteration method \cite{inc2008approx}, homotopy perturbation method \cite{sweilam2007numerical}, Adomian’s decomposition method \cite{jafari2006solving}, homotopy analysis method \cite{hashim2009homotopy} and collocation method \cite{Rawashdeh2006}. While most of the attention has been devoted to the finite difference methods (FDMs), \cite{Lubich1983, Lubich1986, Sanz1988, Sugimoto1991, Metzler2000, Gorenflo2002, Diethelm2004, Langlands2005, Sun2006,  Lin2007, wang2010direct, wang2011fast, Huang2012, Cao2013, zeng2015numerical, zayernouri2016_JCP_Frac_AB_AM},
recent works have focused on exploring the potential efficiency of spectral methods and their inherent global nature as more suitable to discretizing FDEs, see e.g., \cite{Sugimoto1991, Rawashdeh2006, Lin2007, Khader2011, Khader2012, Li2009, Li2010, chen2014generalized, wang2015high, bhrawy2015spectral}.

Two new spectral theories on fractional and tempered fractional Sturm-Liouville problems (TFSLPs) have been recently developed by Zayernouri et al. in \cite{Zayernouri2013, zayernouri2015tempered}. This approach first fractionalizes and then tempers the well-known theory of Sturm-Liouville eigen-problems. The explicit eigenfunctions of TFSLPs are analytically obtained in terms of \textit{tempered Jacobi poly-fractonomials}. Recently, in \cite{Zayernouri_FODEs_2014, Zayernouri14-SIAM-Frac-Delay, Zayernouri14-SIAM-Frac-Advection}, Jacobi poly-fractonomials were successfully employed in developing a series of high-order and efficient Petrov-Galerkin spectral and discontinuous spectral element methods of Galerkin and Petrov-Galerkin projection type for fractional ODEs. To treat nonlinear problems the \textit{collocation} schemes are relatively easy to implement. Khader in \cite{Khader2011} presented a Chebyshev collocation method for the discretization of the space-fractional diffusion equation. More recently, Khader and Hendy \cite{Khader2012} developed a Legendre pseudospectral method for fractional-order delay differential equations. For fast treatment of nonlinear and multi-term fractional PDEs such as the fractional Burgers’ equation, a new spectral method, called \textit{fractional spectral collocation method}, was developed in \cite{Zayernouri14-SIAM-Collocation}. This new class of collocation schemes introduces a new family of fractional Lagrange interpolants, mimicking the structure of the Jacobi poly-fractonomials. For variable-order fractional PDEs, a fast and spectrally accurate collocation method was developed and implemented in \cite{ZayernouriVariableOrder2015JCP}.

Distributed order fractional operators offer a rigorous tool for mathematical modeling of multi-physics phenomena. In this case, the differential order is distributed over a range of values rather than being just a fixed fraction as it is in standard/fractional ODEs/PDEs. There is a rapidly growing interest in the use of fractional derivatives in the construction of mathematical models, which contain distributed order terms of the form
\begin{equation}
	\nonumber
	\int_{\sigma_1}^{\sigma_2} \phi(\sigma) \prescript{*}{a}{\mathcal{D}}_{t}^{\sigma} u(t) d\sigma = f(t), \quad t > a,
\end{equation}
in  the field of uncertainty quantification as the inherent uncertainty of experimental data can be directly incorporated into the differential operators; see \cite{mainardi2007Distributed, diethelm2009Numerical, atanackovic2009existence, atanackovic2009timedis, srokowski2008levy}, for some work on numerical methods. Almost all of the numerical schemes developed for such models are finite-difference methods. While the treatment of fractional differential equations with a fixed fractional order could be memory demanding due to the locality of these methods and their low-accuracy, the main challenge remains the additional effect of the discretization of the distributed order model, which may lead to exceeding computational cost of numerical simulations.

To the best of our knowledge, the first numerical study of distributed order differential equations (DODEs) was performed by Diethelm and Ford in \cite{diethelm2001numerical}, where a two-stage basic framework was developed. In the first stage, the distributed order differentiation term was approximated using a quadrature rule, and in the second stage, a suitable multi-term numerical method was employed. They later performed the corresponding error analysis of the method in \cite{diethelm2009Distributed}. Subsequently, most of the numerical studies have followed the same approach yet they vary in the discretization method in the second stage. The distributed order time-fractional diffusion equation was numerically studied in \cite{ford2014numerical} and the corresponding stability and convergence study of the scheme was provided in \cite{ford2015implicit}. Adding a nonlinear source, \cite{morgado2015numerical} studied the distributed order reaction diffusion equation following the same scheme. In \cite{li2016numerical}, the second stage of the distributed order diffusion equation was established using a reproducing kernel method. The distributed order time fractional diffusion-wave equation was investigated by developing a compact difference scheme in \cite{ye2015compact}. Other numerical studies include: an implicit numerical method of a temporal distributed order and two-sided space-fractional advection-dispersion equation in \cite{hu2015implicit}, high-order difference schemes in \cite{gao2015some}, alternating direction implicit (ADI) difference schemes with the extrapolation method for one-dimensional case in \cite{gao2015one} and two-dimensional problem in \cite{gao2015two}, and an operational matrix technique in \cite{duong2016deterministic}.

In this paper, we first introduce the \textit{distributed Sobolev spaces} and their associated norms. We show their equivalence to the defined left-side and right-side norms as well. By employing Riemann-Liouville derivatives, we define the distributed order differential equation and then obtain its variational form. We develop a Petrov-Galerkin (PG) spectral method following the recent theory of fractional Sturm-Liouville eigen-problems (FSLP) in \cite{Zayernouri2013} and employ the corresponding eigenfunctions, namely the \textit{Jacobi Poly-fractonomial}s of first kind as the bases and the \textit{Jacobi Poly-fractonomial}s of second kind as test functions. We develop a spectrally accurate Gauss-Legendre quadrature rule in the construction of the linear system, where we investigate the stability and error analysis of the scheme. In addition, we construct a spectrally-accurate fractional spectral collocation scheme, where we employ fractional Lagrange interpolants satisfying the Kronecker delta property at the collocation points, and then, we obtain the corresponding fractional differentiation matrices. We demonstrate the computational efficiency of both schemes considering several numerical examples and distribution functions.

The organization of the paper is as follows: section \ref{Sec: Fractional Calculus} provides preliminary definitions along with useful lemmas. We recall fractional Sobolev spaces, and then, introduce their generalization to so called \textit{distributed Sobolev space} and associated norms, which provides the natural setting of our problem in this study. We furthermore obtain some equivalent norms to facilitate the corresponding analysis of our methods. In section \ref{Sec: Variational}, we  derive and discretize the corresponding variational form of the problem and subsequently we prove the stability and convergence rate of the scheme. In addition, we develop a fractional collocation method in section \ref{Sec: Modal: Collocation Method} and test the performance of the two methods in section \ref{Sec: Numerical Simulation}. We conclude the paper with a summary and conclusion.


%
\section{Definitions}
\label{Sec: Fractional Calculus}
%

Let $ \xi \in [-1,1]$. Then, the left-sided and right-sided Riemann-Liouville integral of order $\sigma$,  $n-1 < \sigma \leq n$, $n \in \mathbb{N}$, are defined (see e.g., \cite{Miller93, Podlubny99}) respectively as
\begin{align}
	\label{Eq: left RL integral}
	(\prescript{RL}{-1}{\mathcal{I}}_{\xi}^{\sigma}) u(\xi) = \frac{1}{\Gamma(\sigma)} \int_{-1}^{\xi} \frac{u(s) ds}{(\xi - s)^{n-\sigma} },\quad \xi>-1 ,
	\\
	\label{Eq: right RL integral}
	(\prescript{RL}{\xi}{\mathcal{I}}_{1}^{\sigma}) u(\xi) = \frac{1}{\Gamma(\sigma)} \int_{\xi}^{1} \frac{u(s) ds}{(s - \xi)^{n-\sigma} },\quad \xi<1.
\end{align}
The corresponding left-sided and right-sided fractional derivative of order $\sigma$ are then defined, as  
\begin{align}
	\label{Eq: left RL derivative}
	(\prescript{RL}{-1}{\mathcal{D}}_{\xi}^{\sigma}) u(\xi) = \frac{d^n}{d\xi^n} (\prescript{RL}{-1}{\mathcal{I}}_{\xi}^{n-\sigma} u) (\xi) = \frac{1}{\Gamma(n-\sigma)}  \frac{d^{n}}{d\xi^n} \int_{-1}^{\xi} \frac{u(s) ds}{(\xi - s)^{\sigma+1-n} },\quad \xi >-1 ,
	\\
	\label{Eq: right RL derivative}
	(\prescript{RL}{\xi}{\mathcal{D}}_{1}^{\sigma}) u(\xi) = \frac{(-d)^n}{d\xi^n} (\prescript{RL}{\xi}{\mathcal{I}}_{1}^{n-\sigma} u) (\xi) = \frac{1}{\Gamma(n-\sigma)}  \frac{(-d)^{n}}{d\xi^n} \int_{\xi}^{1} \frac{u(s) ds}{(s - \xi)^{\sigma+1-n} },\quad \xi < 1 ,
\end{align}
respectively. We recall a useful property of the  Riemann-Liouville fractional derivatives \cite{Podlubny99}. Assume that $0<p<1$ and $0<q<1$ and $g(x_L) = 0\,\,\,x>x_L$, then
\begin{equation}
	\label{Eq: Property of RL deriv}
	\prescript{}{x_L}{\mathcal{D}}_{x}^{p+q} g(x) =
	\left( \prescript{}{x_L}{\mathcal{D}}_{x}^{p} \, \prescript{}{x_L}{\mathcal{D}}_{x}^{q} \,\,g\right) (x) 
	= \left( \prescript{}{x_L}{\mathcal{D}}_{x}^{q} \, \prescript{}{x_L}{\mathcal{D}}_{x}^{p} \,\, g \right) (x).
\end{equation}

An alternative approach in defining the fractional derivatives is to begin with the  left-sided Caputo derivatives of order $\sigma$, $n-1 < \sigma \leq n$, $n \in \mathbb{N}$, defined, as
\begin{equation}
	\label{Eq: left Caputo derivative}
	(\prescript{C}{-1}{\mathcal{D}}_{\xi}^{\sigma} u) (\xi) = (\prescript{}{-1}{\mathcal{I}}_{\xi}^{n-\sigma} \frac{d^n u}{d\xi^n}) (\xi) = \frac{1}{\Gamma(n-\sigma)}  \int_{-1}^{\xi} \frac{u^{(n)}(s) ds}{(\xi - s)^{\sigma+1-n} },\quad \xi>-1. 
\end{equation}
By performing an affine mapping from the standard domain $[-1,1]$ to the interval $t \in [a,b]$, we obtain
\begin{eqnarray}
	\label{Eq: RL in xL-xR}
	\prescript{RL}{a}{\mathcal{D}}_{t}^{\sigma} u  &=&  (\frac{2}{b-a})^\sigma (\prescript{RL}{-1}{\mathcal{D}}_{\xi}^{\sigma} \, u )(\xi), 
	\\ 
	\label{Eq: Caputo in xL-xR}
	\prescript{C}{a}{\mathcal{D}}_{t}^{\sigma} u  &=&  (\frac{2}{b-a})^\sigma (\prescript{C}{-1}{\mathcal{D}}_{\xi}^{\sigma} \, u) (\xi).
\end{eqnarray} 
Hence, we can perform the operations in the standard domain only once for any given $\sigma$ and efficiently utilize them on any arbitrary interval without resorting to repeating the calculations. Moreover, the corresponding relationship between the Riemann-Liouville and Caputo fractional derivatives in $[a,b]$ for any $\sigma \in (0,1)$ is given by 
\begin{equation}
	\label{Eq:  Caputo vs. Riemann}
	(\prescript{RL}{a}{\mathcal{D}}_{t}^{\sigma} \, u) (t)  =  \frac{ u(a)}{\Gamma(1-\sigma) (t-a)^{\sigma}}  +   (\prescript{C}{a}{\mathcal{D}}_{t}^{\sigma} \, u) (t).
\end{equation}

\vspace{0.2cm}
\begin{lem}
	\label{Lem: Frac derive of Poly-fractonomials}
	Let $\sigma , \mu > 0$. The fractional derivative of the Jacobi poly-fractonomials, \cite{Zayernouri2013}, of first ($i=1$) and second kind ($i=2$) are given by
	\begin{eqnarray}
		\label{Eq: in Lemma 1Presenving structure}
		\prescript{RL}{} {\mathcal{D}}_{}^{\sigma} 
		\Big\lbrace
		\prescript{(i)}{}{ \mathcal{P}}_{n}^{\mu}(\xi)
		\Big\rbrace  &=& \frac{\Gamma(n +\mu)}{\Gamma(n +\mu - \sigma)}  \prescript{(i)}{}{ \mathcal{P}}_{n}^{\, \mu - \sigma}(\xi),
	\end{eqnarray}
	and are also of Jacobi poly-fractonomial type, where $\prescript{RL}{} {\mathcal{D}}_{}^{\sigma}  \equiv \prescript{RL}{-1} {\mathcal{D}}_{x}^{\sigma}$ when $i=1$, $\prescript{RL}{} {\mathcal{D}}_{}^{\sigma}  \equiv \prescript{RL}{x} {\mathcal{D}}_{1}^{\sigma}$ when $i=2$.
\end{lem}
\begin{proof}
	See Appendix \eqref{Sec: App. proof frac derive of Poly-fractonomials}.

\end{proof}
\vspace{0.2cm}
\begin{rem}
	Lemma \ref{Lem: Frac derive of Poly-fractonomials} shows that the structure of Jacobi poly-fractonomials is preserved under the action of fractional derivatives. Moreover, we note that when $\sigma = \mu$ in Lemma \ref{Lem: Frac derive of Poly-fractonomials}, the fractional derivatives of Jacobi poly-fractonomials are obtained in terms of Legendre polynomials, which has been reported in \cite{Zayernouri2013}.
\end{rem}
%

%
\subsection{Fractional Sobolev Spaces}
\label{Sec: Fractional Sobolev Spaces}
%

By $H^s(\mathbb{R})$, $s \geq 0$, we denote the fractional Sobolev space on $\mathbb{R}$, defined as 
\begin{equation}
	\label{Eq: Frac Sobolev on R}
	H^s(\mathbb{R}) = \{ v \in L^2(\mathbb{R}) \vert \,\, (1+\vert \omega \vert^2 )^{\frac{s}{2}}\mathcal{F}(v)(\omega) \in L^2(\mathbb{R})   \},
\end{equation}
which is endowed with the norm 
\begin{equation}
	\label{Eq: Norm Frac Sobolev on R}
	\Vert \cdot \Vert_{s, \mathbb{R}}  =  \Vert  (1+\vert \omega \vert^2 )^{\frac{s}{2}}\mathcal{F}(\cdot)(\omega)   \Vert_{L^2(\mathbb{R})},
\end{equation}
where $\mathcal{F}(v)$ represents the Fourier transform of $v$. Subsequently, we denote by $H^s(I)$, $s \geq 0$ the fractional Sobolev space on any finite closed interval $I$, defined as 
\begin{equation}
	\label{Eq: Frac Sobolev on Interval}
	H^s(I) = \{ v \in L^2(\mathbb{R}) \vert \,\,\, \exists \tilde{v} \in H^s(\mathbb{R})\,\,\, s.t.\,\,\, \tilde{v}|_{I} = v   \},
\end{equation}
with the norm
\begin{equation}
	\label{Eq: NormFrac Sobolev on Interval}
	\Vert \cdot \Vert_{s, I}  =  \inf_{\tilde{v} \in H^s(\mathbb{R}), \tilde{v}|_{I}=(\cdot)  } \Vert \tilde{v} \Vert_{s, \mathbb{R}}.
\end{equation}
We note that the definition of $H^s(I)$ and the corresponding norm relies on the Fourier transformation of the function. Other useful norms associated with $H^s(I)$, e.g., when $I = [x_L,x_R]$, have been also introduced in \cite{Li2009}, 
\begin{align}
	\label{Eq: LeftEquivallent NormFrac Sobolev on Interval}
	\Vert \cdot \Vert_{l, s, I}  = \Big(  \Vert \cdot \Vert^2_{L^2(I)} +   \Vert  \prescript{RL}{x_L}{\mathcal{D}}_{x}^{\mu}(\cdot)  \Vert^2_{L^2(I)}    \Big)^{\frac{1}{2}},
	\\
	\label{Eq: RightEquivallent NormFrac Sobolev on Interval}
	\Vert \cdot \Vert_{r, s, I}  = \Big(  \Vert \cdot \Vert^2_{L^2(I)} +   \Vert  \prescript{RL}{x}{\mathcal{D}}_{x_R}^{\mu}(\cdot)  \Vert^2_{L^2(I)}    \Big)^{\frac{1}{2}},
\end{align}
such that the left-side $\Vert \cdot \Vert_{l, s, I}$, the right-sided $\Vert \cdot \Vert_{r, s, I} $, and $\Vert \cdot \Vert_{s, I} $ are shown to be equivalent. 

Next, let $\phi \in L^1(\,[\alpha_{min},\alpha_{max}]\,)$, $0 \leq \alpha_{min} < \alpha_{max} $, be nonnegative. By $\prescript{\phi}{}{\mathcal H}(\mathbb{R})$, we denote the \textit{distributed} fractional Sobolev space on $\mathbb{R}$, defined as 
%
\begin{equation}
	\label{Eq: dis. Frac Sobolev on R}
	\prescript{\phi}{}{\mathcal H}(\mathbb{R}) = \{ v \in L^2(\mathbb{R}) \vert \,\, \int_{\alpha_{min}}^{\alpha_{max}} \left[\phi(\alpha)(1+\vert \omega \vert^2 )^\alpha\right]^{\frac{1}{2}} \,\mathcal{F}(v)(\omega) \,\, d\alpha \,\in L^2(\mathbb{R})   \},
\end{equation}
which is endowed with the norm 
\begin{equation}
	\label{Eq: Norm dis Frac Sobolev on R}
	\Vert \cdot \Vert_{\phi, \mathbb{R}}  =  
	\left(
	\int_{\alpha_{min}}^{\alpha_{max}} \phi(\alpha) \,\, \norm{ \, (1+\vert \omega \vert^2 )^{\frac{\alpha}{2}}\mathcal{F}(\cdot)(\omega) \,  }^2_{L^2(\mathbb{R})} \,\, d\alpha
	\right)^{\frac{1}{2}}.
\end{equation}

Subsequently, we denote by $\prescript{\phi}{}{\mathcal H}(I)$ the \textit{distributed} fractional Sobolev space on the finite closed interval $I$, defined as 
\begin{equation}
	\label{Eq: dis. Frac Sobolev on Interval}
	\prescript{\phi}{}{\mathcal H}(I) = \{ v \in L^2(\mathbb{R}) \vert \,\,\, \exists \tilde{v} \in \prescript{\phi}{}{\mathcal H}(\mathbb{R})\,\,\, s.t.\,\,\, \tilde{v}|_{I} = v   \},
\end{equation}
with the norm
\begin{equation}
	\label{Eq: dis. NormFrac Sobolev on Interval}
	\Vert \cdot \Vert_{\phi, I}  =  \inf_{\tilde{v} \in \prescript{\phi}{}{\mathcal H}(\mathbb{R}), \tilde{v}|_{I}=(\cdot)  } \Vert \tilde{v} \Vert_{\phi, \mathbb{R}}.
\end{equation}
Moreover, we introduce the following useful norms, associated with $\prescript{\phi}{}{\mathcal H}(I)$:
\begin{equation}
	\label{Eq: LeftEquivallent dis. NormFrac Sobolev on Interval}
	\Vert \cdot \Vert_{l, \phi, I}  =  
	\left(
	\Vert \cdot \Vert^2_{L^2(I)} + 
	\int_{\alpha_{min}}^{\alpha_{max}} \phi(\alpha) \,\, \norm{ \, \prescript{RL}{x_L}{\mathcal{D}}_{x}^{\alpha}(\cdot)  \, }^2_{L^2(I)} \,\, d\alpha
	\right)^{\frac{1}{2}},
\end{equation}
and
\begin{equation}
	\label{Eq: RightEquivallent dis. NormFrac Sobolev on Interval}
	\Vert \cdot \Vert_{r, \phi, I}  = 
	\left(
	\Vert \cdot \Vert^2_{L^2(I)} + 
	\int_{\alpha_{min}}^{\alpha_{max}} \phi(\alpha) \,\, \norm{ \, \prescript{RL}{x_L}{\mathcal{D}}_{x}^{\alpha}(\cdot)  \, }^2_{L^2(I)} \,\, d\alpha
	\right)^{\frac{1}{2}}.
\end{equation}
\begin{figure}[!]
	\centering
	\includegraphics[width=1.1\linewidth]{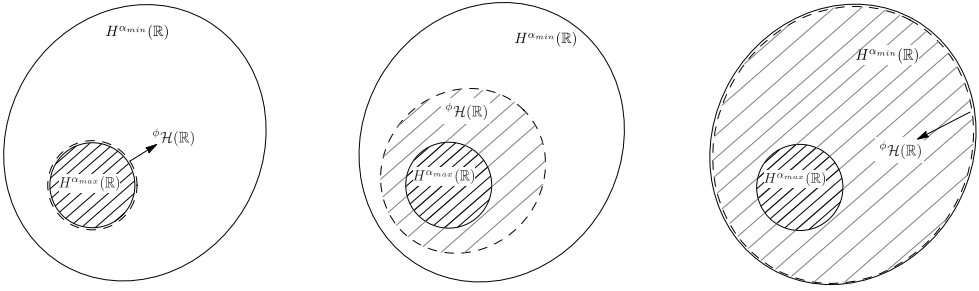}
	\caption{\small{Schematic of distributed fractional Sobolev space $\prescript{\phi}{}{\mathcal H}(\mathbb{R})$: (left) $\phi = \delta(\alpha - \alpha_{max})$ hence $\prescript{\phi}{}{\mathcal H}(\mathbb{R}) = H^{\alpha_{max}}(\mathbb{R})$; (middle) $\phi$ defined on a compact support in $[\alpha_{min} , \alpha_{max}]$, hence, $\prescript{\phi}{}{\mathcal H}(\mathbb{R}) \supset H^{\alpha_{max}}(\mathbb{R})$; (right) $\phi = \delta(\alpha - \alpha_{min})$, where $\prescript{\phi}{}{\mathcal H}(\mathbb{R}) = H^{\alpha_{min}}(\mathbb{R})$.}}
	\label{Fig: Fractional Sobolve Space}
\end{figure}

We note that when $\phi > 0$ is continuous in $I$, $\prescript{\phi}{}{\mathcal H}(\mathbb{R})$ is equivalent to $H^{\alpha_{max}}(\mathbb{R})$. However, in general, the choice of $\phi$ can arbitrarily confine the domain of integration in practice. In other words, $\alpha_{min}$ and $\alpha_{max}$ are only the theoretical lower and upper terminals in the definition of distributed order fractional derivative. For instance, in a distributed sub-diffusion problem, the temporal derivative is associated with $\alpha_{min} = 0$ and $\alpha_{max} = 1$, and in a super-diffusion problem, the theoretical upper terminal $\alpha_{max} =2$. In this study we particularly aim to let $\phi$ be defined in any possible subset of the interval $[\alpha_{min} , \alpha_{max}]$. Hence, in each realization of a physical process (e.g. sub- or super-diffusion) $\phi$ can be obtained from data, where the theoretical setting of the problem remains invariant yet requiring the solution to have less regularity (since $\prescript{\phi}{}{\mathcal H}(\mathbb{R}) \supset  H^{\alpha_{max}}(\mathbb{R})$ in general, see Fig.\ref{Fig: Fractional Sobolve Space}).

In the following theorem, we prove the equivalence (shown by the notation $\sim$) of the aforementioned norms.
\begin{theorem}
	\label{Thm: dis. norm equivalency}
	Let $\phi \in L^1(\, [\alpha_{min} , \alpha_{max}] \,) $ be non-negative. Then, the norms $\Vert \cdot \Vert_{\phi, I}$, $\Vert \cdot \Vert_{l, \phi, I}$, and $\Vert \cdot \Vert_{r, \phi, I}$ are equivalent.
\end{theorem}
\begin{proof}
	See Appendix \eqref{Sec: App. proof dis. norm equivalency}.
\end{proof}
\vspace{0.2 cm}
%
\vspace{0.2 cm}
\begin{lem}
	\label{Lem: left frac proj}
	\cite{Li2009}: For all $0< \alpha \leq 1$, if $u \in H^1([a,b])$ such that $u(a)=0$, and $w \in H^{\alpha/2}([a,b])$, then
	\begin{equation}
		\label{Eq: left frac proj}
		( \prescript{}{a}{ \mathcal{D}}_{s}^{\,\,\alpha} u, w )_{\Omega} =  (\, \prescript{}{a}{ \mathcal{D}}_{s}^{\,\,\alpha/2} u \,,\, \prescript{}{s}{ \mathcal{D}}_{b}^{\,\,\alpha/2} w\, )_{\Omega},
	\end{equation}
	where $(\cdot , \cdot)_{\Omega}$ represents the standard inner product in $\Omega=[a,b]$. 
	%
\end{lem}

\begin{lem}
	\label{Lem: fractional integ-by-parts 1 and 2}
	Let $1/2 < \mu < 1$, $a$ and $b$ be arbitrary finite or infinite real numbers. Assume $u \in H^{2\mu}(a,b)$ such that $u(a)=0$, also $\prescript{}{x}{\mathcal{D}}_{b}^{\mu}v$ is integrable in $(a,b)$ such that $v(b) = 0$. Then
	\begin{equation}
		\label{Eq: FracIntegByParts}
		( \prescript{}{a}{\mathcal{D}}_{x}^{2\mu} u \,,\,v ) = ( \prescript{}{a}{\mathcal{D}}_{x}^{\mu} u \,,\,\prescript{}{x}{\mathcal{D}}_{b}^{\mu} v ). 
	\end{equation}
\end{lem}
\begin{proof}
	See Appendix \eqref{Sec: App. proof fractional integ-by-parts 1 and 2}.
\end{proof}
\begin{rem}
	Unlike other existing proofs (e.g., see Proposition 1 in \cite{zhang2010galerkin}), our proof requires $v(x)$ to only vanish at the right boundary (note that $v(a)$ can be non-zero), moreover, we only require the $\mu$-th derivative (rather than the first derivative) of $v(x)$ to be integrable in $(a,b)$.
\end{rem}

%

\subsection{Problem Definition}
\label{Sec: problem definition}

Following \cite{Atanackovic2009distributional}, let $\alpha \mapsto \phi(\alpha)$ be a continuous mapping in $[\alpha_{min} , \alpha_{max}]$. Then, we define the distributed order fractional derivative as
\begin{equation}
	\label{Eq: Distributed derivative-I}
	\prescript{D}{}{\mathcal{D}}_{\phi}^{} u(t)= 
	\int_{\alpha_{min}}^{\alpha_{max}} \phi(\alpha) \prescript{*}{a}{\mathcal{D}}_{t}^{\alpha} u(t) d\alpha, \quad t > a,
\end{equation}
where $\prescript{*}{a}{\mathcal{D}}_{t}^{\alpha}$ denotes Riemann-Liouville fractional derivative of order $\alpha$. Next, we aim to solve the following differential equation of distributed order: 
\begin{eqnarray}
	\label{Eq: distributional order problem}
	\prescript{D}{}{\mathcal{D}}_{\phi}^{} u(t) &=& f(t; u), \quad \forall t \in (0,T],
	\\
	\label{Eq: distributional ICs-1}
	u(0) &=& 0, \quad (\alpha_{max}) \in (0,1],
	\\
	\label{Eq: distributional ICs-2}
	u(0) =\frac{du}{dt}|_{t=0} &=& 0, \quad  (\alpha_{max})  \in (1,2].
\end{eqnarray}
In the sequel, we present different approaches to discretize the aforementioned differential operator. Due to \eqref{Eq:  Caputo vs. Riemann}, the Caputo and Riemann-Liouville fractional derivatives of order $\alpha \in (0,1)$ coincide with each other when $u(a) = 0$. Therefore, in this study, we employ the definition of the distributed fractional derivatives of Riemann-Liouville sense and remove the pre-superscript $RL$ for simplicity.

%
\section{Variational Formulation}
\label{Sec: Variational}
%
In order to obtain the variational form we multiply \eqref{Eq: distributional order problem} by a proper test function $v$ (defined later) and integrate over the computational domain:

\begin{equation}
	\label{Eq: variational form-1}
	\int_{\Omega}^{} \prescript{D}{}{\mathcal{D}}_{\phi}^{} u(t) \, v(t) \, d\Omega = \int_{\Omega}^{} f(t; u) \, v(t) \, d\Omega .
\end{equation}

\noindent Using the definition of distributed order fractional derivatives defined in \eqref{Eq: Distributed derivative-I} we get

\begin{equation}
	\label{Eq: variational form-2}
	\int_{\alpha_{min}}^{\alpha_{max}} \phi(\alpha) \int_{\Omega}^{} \prescript{}{0}{\mathcal{D}}_{t}^{\alpha} u(t) \, v(t) \, d\Omega \, d\alpha = \int_{\Omega}^{} f(t; u) \, v(t) \, d\Omega ,
\end{equation}

\noindent where

\begin{equation*}
	\forall \alpha \in (\alpha_{min},\alpha_{max}),  \quad \int_{\Omega}^{} \prescript{}{0}{\mathcal{D}}_{t}^{\alpha} u \, v \, d\Omega = \left(\prescript{}{0}{\mathcal{D}}_{t}^{\alpha} u , v\right)_{\Omega}
\end{equation*}

\noindent denotes the well-known $L^2$-inner product. Given the initial conditions \eqref{Eq: distributional ICs-1} and/or \eqref{Eq: distributional ICs-2} and by  Lemmas \eqref{Lem: left frac proj}, we define the bilinear form associated with $\alpha \in (\alpha_{min},\alpha_{max})$ as

\begin{eqnarray}
	\label{Eq: bilinear form}
	a^{(\alpha)} (u,v)  =  \left( \prescript{}{0}{\mathcal{D}}_{t}^{\alpha} u , v \right)_{\Omega} 
	=  ( \prescript{}{0}{\mathcal{D}}_{t}^{\alpha/2} u , \prescript{}{t}{\mathcal{D}}_{T}^{\alpha/2} v )_{\Omega}.
\end{eqnarray}
We choose $v$ such that $v(T) = 0$ and $\prescript{}{t}{\mathcal{D}}_{T}^{\alpha/2} v$ is integrable in $\Omega \,\,$  $\forall \alpha \in (\alpha_{min},\alpha_{max})$. Moreover, let $U$ be the solution space, defined as

\begin{eqnarray}
	\label{Eq: U definition}
	U = \lbrace u \in L^2(\Omega) : & & \sqrt{ \int_{\alpha_{min}}^{\alpha_{max}}  \phi(\alpha) \norm{\prescript{}{0}{\mathcal{D}}_{t}^{\alpha/2} u}^2_{L^2(\Omega)} d\alpha } <  \infty 
	\\
	& s.t. \quad
	&
	\begin{array} {ll}
		u(0) = 0 & \text{if} \,\, \alpha_{max} \in (0,1]
		\\ 
		u(0) = \frac{du}{dt}|_{t=0} = 0 & \text{if} \,\, \alpha_{max} \in (1,2]
	\end{array}
	\nonumber
	\rbrace
\end{eqnarray}

\noindent and let $V$ be the test function space given by

\begin{equation}
	\label{Eq: V definition}
	V = \lbrace v \in L^2(\Omega) : \sqrt{ \int_{\alpha_{min}}^{\alpha_{max}}  \phi(\alpha) \norm{\prescript{}{t}{\mathcal{D}}_{T}^{\alpha/2} v}^2_{L^2(\Omega)} d\alpha } <  \infty \quad s.t. \quad  v(T) = 0 \rbrace .
\end{equation}
The problem thus reads as: find $u \in U$ such that $a(u,v) = l(v), \quad \forall v \in V$ where 

\begin{equation}
	\label{Eq: distributed bilinear form}
	a(u,v) := \int_{\alpha_{min}}^{\alpha_{max}} \phi(\alpha) \,\, a^{(\alpha)}(u,v) \, d\alpha
\end{equation}
denotes the \textit{distributed} bilinear form and $l(v) := (f,v)_{\Omega}$.

%
\subsection{Petrov-Galerkin Method: Modal Expansion}
\label{Sec: Modal: PG Method}
%

In the Petrov-Galerkin (PG) method, we follow the recent theory of fractional Sturm-Liouville eigen-problems (FSLP) in \cite{Zayernouri2013} and employ the corresponding eigenfunctions, known as the \textit{Jacobi Poly-fractonomial}s (of first kind) given in the standard domain $[-1,1]$ by
\begin{equation}
	\label{Eq: RegularBasis I}
	\prescript{(1)}{}{ \mathcal{P}}_{n}^{\,\,\mu}(\xi) = (1+\xi)^{\mu} P_{n-1}^{-\mu,\mu} (\xi),\quad \xi \in [-1,1],
\end{equation}

\noindent as \textit{non-polynomial} basis functions consisting of a fractional term multiplied by the Jacobi polynomial $P_{n-1}^{-\mu,\mu} (\xi)$, hence we employ \eqref{Eq: RegularBasis I} in construction of a basis to formulate a projection type scheme, namely \textit{modal} expansion. We represent the solution in terms of the elements of the basis space $U_N$ given as follows

\begin{align}
	\label{Eq: UN}
	U_N = 
	span \big\{ \prescript{(1)}{}{ \mathcal{P}}_{n}^{\mu}(\xi), \quad \xi \in [-1,1], \quad n= 1,2,\cdots , N \big\},
\end{align}

\noindent via the poly-fractonomial \textit{modal} expansion as

\begin{equation}
	\label{Eq: Approx SolN Modal}
	u_N(\xi) = \sum_{n=1}^{N} c_n  \prescript{(1)}{}{ \mathcal{P}}_{n}^{\mu}(\xi),
\end{equation}
in which $\mu$ is to be fixed as a fractional parameter \textit{a priori} depending on the range of distribution order interval, i.e.\, $\mu \in (0,1)$ if $\alpha_{max} \in (0,1]$ and $\mu \in (1,2)$ if $\alpha_{max} \in (1,2]$. It can also be tunned to capture possible singularities in the exact solution if some knowledge about that is available. 

Moreover, in the PG scheme, we employ another space of test functions $V_N$, however of the same dimension, given by
\begin{align}
	\label{Eq: VN}
	V_N = 
	span \big\{ \prescript{(2)}{}{ \mathcal{P}}_{k}^{\mu}(\xi), \quad \xi \in [-1,1], \quad k=1,2,\cdots , N \big\} , 
\end{align}
in which $ \prescript{(2)}{}{ \mathcal{P}}_{k}^{\mu}(\xi) = (1-\xi)^{\mu} P^{\mu, -\mu}_{k-1}(\xi)$ denotes the Jacobi poly-fractonomial of second kind, which is the explicit eigenfunction of fractional Sturm-Liouville problem of second kind in \cite{Zayernouri2013}. 

It should be noted that since $\phi(\alpha) \geq 0$ and $\int_{\alpha_{min}}^{\alpha_{max}} \phi(\alpha) \, d\alpha = 1$ it is not difficult to see that $U_N \subset U$ and $V_N \subset V$ when $\mu$ is chosen properly. Therefore, the bilinear form \eqref{Eq: bilinear form} reduces to the \textit{discrete} bilinear form

\begin{equation}
	\label{Eq: projected bilinear form}
	a_h^{\alpha} (u_N,v_N) = ( \prescript{}{0}{\mathcal{D}}_{t}^{\alpha/2} u_N , \prescript{}{t}{\mathcal{D}}_{T}^{\alpha/2} v_N )_{\Omega}
\end{equation}
and thus the problem reads as: find $u_N \in U_N$ such that

\begin{equation}
	\label{Eq: discrete distributed bilinear form}
	a_h(u_N,v_N) = l_h(v_N), \quad \forall v_N \in V_N,
\end{equation}
where $a_h(u_N,v_N) := \int_{\alpha_{min}}^{\alpha_{max}} \phi(\alpha) \,\, a^{\alpha}_h(u_N,v_N) \, d\alpha \,\,$denotes the \textit{discrete} distributed bilinear form and $l_h(v_N) := (f,v_N)_{\Omega}$ represents the load vector. 

By substituting the expansion \eqref{Eq: Approx SolN Modal},  choosing $v_N = \prescript{(2)}{}{ \mathcal{P}}_{k}^{\mu}(\xi) \in V_N$, $k=1,2,\cdots,N$ and using \eqref{Eq: RL in xL-xR}, the discrete distributed bilinear form in \eqref{Eq: discrete distributed bilinear form} can be written as

\begin{align}
	\label{Eq: discrete distributed bilinear form UN rep}
	&\int_{\alpha_{min}}^{\alpha_{max}} \phi(\alpha) \,\,
	\sum_{n=1}^{N} c_n (\frac{2}{T})^{\alpha}
	\Big( \prescript{}{-1}{\mathcal{D}}_{\xi}^{\alpha/2} [\prescript{(1)}{}{ \mathcal{P}}_{n}^{\mu}(\xi)]
	\, , \, 
	\prescript{}{\xi}{\mathcal{D}}_{1}^{\alpha/2} [
	\prescript{(2)}{}{ \mathcal{P}}_{k}^{\mu}(\xi)]\, \Big)_{\Omega}  \, 
	d\alpha.
\end{align}
From Lemma \ref{Lem: Frac derive of Poly-fractonomials}, we have
\begin{eqnarray}
	\label{Eq: from lemma frac derive 1}
	\prescript{}{-1}{\mathcal{D}}_{\xi}^{\alpha/2} [ \prescript{(1)}{}{ \mathcal{P}}_{n}^{\mu}(\xi) ]
	&=&
	\frac{\Gamma(n +\mu)}{\Gamma(n + \eta)}  \prescript{(1)}{}{ \mathcal{P}}_{n}^{\eta}(\xi)
	\\ 
	\label{Eq: from lemma frac derive 2}
	\prescript{RL}{\xi}{\mathcal{D}}_{1}^{\alpha/2}[
	\prescript{(2)}{}{ \mathcal{P}}_{k}^{\mu}(\xi) ],
	&=&
	\frac{\Gamma(k +\mu)}{\Gamma(k + \eta)}  \prescript{(2)}{}{ \mathcal{P}}_{k}^{\eta}(\xi),
\end{eqnarray}
where $\eta = \mu - \alpha/2$. Thus, by changing the order of summation, the integral \eqref{Eq: discrete distributed bilinear form UN rep} takes the form
\begin{align}
	\label{Eq: discrete distributed bilinear form UN rep 2}
	&
	\sum_{n=1}^{N} c_n \,\,
	\int_{\alpha_{min}}^{\alpha_{max}} 
	\phi(\alpha) \,\, (\frac{2}{T})^{\alpha} \,
	\frac{\Gamma(n +\mu)}{\Gamma(n + \eta)}  \frac{\Gamma(k +\mu)}{\Gamma(k + \eta)}  
	\Big( 
	\prescript{(1)}{}{ \mathcal{P}}_{n}^{\eta}(\xi)
	\, , \, 
	\prescript{(2)}{}{ \mathcal{P}}_{k}^{\eta}(\xi)
	\Big)_{\Omega}  \, 
	d\alpha
	\\
	\nonumber
	=&
	\sum_{n=1}^{N} c_n \,\,
	\int_{\alpha_{min}}^{\alpha_{max}} 
	\phi(\alpha) \, (\frac{2}{T})^{\alpha} \,
	\frac{\Gamma(n +\mu)}{\Gamma(n + \eta)}  \frac{\Gamma(k +\mu)}{\Gamma(k + \eta)}  
	\int_{-1}^{1}
	\prescript{(1)}{}{ \mathcal{P}}_{n}^{\eta}(\xi) \,\,
	\prescript{(2)}{}{ \mathcal{P}}_{k}^{\eta}(\xi) \,\,
	d\xi
	\, d\alpha,
\end{align}
where by changing the order of integrations we get
\begin{align}
	\label{Eq: discrete distributed bilinear form UN rep 3}
	& a_h(u_N,v_N)
	\\
	\nonumber
	=&
	\sum_{n=1}^{N} c_n \,\,
	\int_{-1}^{1}
	\left[
	\int_{\alpha_{min}}^{\alpha_{max}} 
	\phi(\alpha) \, (\frac{2}{T})^{\alpha} \,
	\frac{\Gamma(n +\mu)}{\Gamma(n + \eta)}  \frac{\Gamma(k +\mu)}{\Gamma(k + \eta)}  
	\prescript{(1)}{}{ \mathcal{P}}_{n}^{\eta}(\xi) \,\,
	\prescript{(2)}{}{ \mathcal{P}}_{k}^{\eta}(\xi) \,\,
	\, d\alpha
	\right] \,
	d\xi.
\end{align}
\vspace{0.1 in}
\begin{theorem}[Spectrally/Exponentially Accurate Quadrature Rule in $\alpha$-Dimension]
	\label{Thm: quadrature convergence}
	
	\noindent Part A: 
	$\forall \xi = \xi_0 \in [-1,1]$ fixed, and $\forall n \in \varmathbb{N} \cup \lbrace 0 \rbrace$, the Jacobi polynomial $P^{\pm \alpha , \mp \alpha}_{n}(\xi_0)$ is a polynomial of order $n$ in $\alpha$.
	
	\vspace{0.1 in}
	\noindent Part B:
	Let $\phi \in H^r([\alpha_{min},\alpha_{max}])$, $r>0$. Then $\forall \mu \geq \alpha_{max}/2$
	\begin{align*}
		\bigg|
		&\int_{\alpha_{min}}^{\alpha_{max}}  
		\phi(\alpha) \, (\frac{2}{T})^{\alpha} \,
		\frac{\Gamma(n +\mu)}{\Gamma(n + \eta)}  \frac{\Gamma(k +\mu)}{\Gamma(k + \eta)}  
		\prescript{(1)}{}{ \mathcal{P}}_{n}^{\eta}(\xi) \,\,
		\prescript{(2)}{}{ \mathcal{P}}_{k}^{\eta}(\xi) \,\,
		d\alpha \,\,  -
		\\
		&
		\quad
		\sum_{q=1}^{Q} 
		\, \widetilde{w}_q \, 
		\phi(\alpha_q) \, (\frac{2}{T})^{\alpha_q} \,
		\frac{\Gamma(n +\mu)}{\Gamma(n + \eta_q)}  \frac{\Gamma(k +\mu)}{\Gamma(k + \eta_q)}  
		\prescript{(1)}{}{ \mathcal{P}}_{n}^{\eta_q}(\xi) \,\,
		\prescript{(2)}{}{ \mathcal{P}}_{k}^{\eta_q}(\xi) \,\,
		\bigg|
		\leq
		\mathscr{C} \, Q^{-r} \norm{\phi}_{H^r([\alpha_{min} , \alpha_{max}])},
	\end{align*}
	where $\mathscr{C} > 0$, $\phi_N (\alpha) = \sum_{n=0}^{N} \,\, \widetilde{\rho}_n \, P_n(\alpha)$ denotes the polynomial expansion of $\phi(\alpha)$, and $\lbrace \alpha_q , \widetilde{w}_q \rbrace \bigg|^{Q}_{q=1}$ represents the set of Gauss-Legendre quadrature points and weights.
	
	\vspace{0.5 cm}
	\noindent Part C: If $\phi(\alpha)$ is smooth, the quadrature rule in $\alpha$-dimension becomes exponentially accurate in $Q$.

\end{theorem}
\begin{proof}
	See Appendix \eqref{Sec: App. proof quadrature convergence}.
	
\end{proof}
%

By theorem \eqref{Thm: quadrature convergence} and performing an affine mapping from $[\alpha \in \alpha_{min},\alpha_{max}]$ to the standard domain $\alpha_{st} \in [-1,1]$, the inner integral in \eqref{Eq: discrete distributed bilinear form UN rep 3} can be evaluated with spectral accuracy by employing a Gauss-Legendre quadrature rule. Then by changing the order of summation and integral $a_h(u_N,v_N) = l_h(v_N)$ can be written as:
\begin{align}
	\label{Eq: St. discrete distributed bilinear form}
	\sum_{n=1}^{N} \,  
	\sum_{j=1}^{Q} \,
	c_n \, J_\alpha \,  w_j \, \phi_j \, (\frac{2}{T})^{\alpha_j} \,
	\frac{\Gamma(n +\mu)}{\Gamma(n + \eta_j)} \frac{\Gamma(k +\mu)}{\Gamma(k + \eta_j)}
	\int_{-1}^{1}
	\prescript{(1)}{}{ \mathcal{P}}_{n}^{\eta_j}(\xi) \,\,
	\prescript{(2)}{}{ \mathcal{P}}_{k}^{\eta_j}(\xi) \,\,
	d\xi
	&
	\\
	\nonumber
	= \Big( f\, , \,  \prescript{(2)}{}{ \mathcal{P}}_{k}^{\mu}(\xi) \Big)_{\Omega}, \quad\quad k=1,2,...,N  ,
	&
\end{align}
where $J_\alpha = \frac{d\alpha}{d\alpha_{st}} = \frac{(\alpha_{max} - \alpha_{min})}{2}$ is constant and $\alpha_j = \alpha(\alpha_{{st}_j})$, $\phi_j = \phi( \alpha( \alpha_{{st}_j}))$, $\eta_j = \mu - \alpha_j/2$ and $\alpha_{{st}_j}$ and $w_j$ are the quadrature points and weights respectively.
%
%
%
%
%
%
The linear system is then constructed as
\begin{equation}
	\label{Eq: linear sys PG}
	\mathcal{S} \vec{c} = \vec{\mathcal{F}},
\end{equation}
in which the entries of the stiffness matrix $\mathcal{S}$ and force vector $\vec{\mathcal{F}}$ are given by
\begin{equation}
	\label{Eq: Stifness linear sys PG}
	\mathcal{S}_{kn} =  
	\sum_{j=0}^{Q-1} 
	J_\alpha \,  w_j \, \phi_j \, (\frac{2}{T})^{\alpha_j} \,\,
	C_{kn} \,
	\int_{-1}^{1}
	\prescript{(1)}{}{ \mathcal{P}}_{n}^{\eta_j}(\xi)
	\prescript{(2)}{}{ \mathcal{P}}_{k}^{\eta_j}(\xi)
	d\xi
\end{equation}
and
\begin{equation}
	\label{Eq: Force linear sys PG}
	\mathcal{F}_{k} = 
	\int_{-1}^{1}
	f(\xi)
	\prescript{(2)}{}{ \mathcal{P}}_{k}^{\mu}(\xi) \,\,
	d\xi
\end{equation}
respectively, where
\begin{equation*}
	C_{kn} \equiv \frac{\Gamma(n +\mu)}{\Gamma(n + \eta_j)}
	\frac{\Gamma(k + \mu)}{\Gamma(k + \eta_j)}.
\end{equation*}
\begin{rem}
	For each fixed $j$ and given the structure of $\prescript{(1)}{}{ \mathcal{P}}_{n}^{\eta_j}(\xi)$ and $\prescript{(2)}{}{ \mathcal{P}}_{k}^{\eta_j}(\xi)$, the above integrations take the form 
	\begin{eqnarray}
		\nonumber
		\int_{-1}^{1}
		\prescript{(1)}{}{ \mathcal{P}}_{n}^{\eta_j}(\xi)
		\prescript{(2)}{}{ \mathcal{P}}_{k}^{\eta_j}(\xi)
		d\xi & = & 
		\int_{-1}^{1}
		(1-\xi)^{\eta_j} (1+\xi)^{\eta_j} 
		\,\,
		P^{\eta_j , -\eta_j}_{k-1}(\xi) \,\,
		P^{-\eta_j , \eta_j}_{n-1}(\xi) \,\,
		d\xi,
		\\
		\nonumber
		\int_{-1}^{1}
		f(\xi) \prescript{(2)}{}{ \mathcal{P}}_{k}^{\mu}(\xi) \,\, 
		d\xi & = & 
		\int_{-1}^{1} (1-\xi)^{\mu} \,\, f(\xi) \,\, P^{\mu , -\mu}_{k-1}(\xi) \,\,
		d\xi,
	\end{eqnarray}
	and therefore, the full stiffness matrix $\mathcal{S}$ and vector $\vec{\mathcal{F}}$ can be constructed accurately using a proper Gauss-Lobatto-Jacobi rule corresponding to the weight function $(1-\xi)^{\eta_j} (1+\xi)^{\eta_j} $ and $(1-\xi)^{\eta_j}$ respectively.
\end{rem}
%

%
\subsection{Discrete Stability Analysis}
\label{Sec: discrete stability analysis} 
%

In this section, we investigate the stability of the numerical scheme, developed based on the aforementioned choice of solution and test function space considering the bilinear form in \eqref{Eq: discrete distributed bilinear form}.
\vspace{0.2 cm}
\begin{theorem}
	\label{Thm: Stability of scheme}
	The scheme \eqref{Eq: discrete distributed bilinear form} is stable and the following inequality holds
	\begin{equation*}
		\underset{u_N \in U_N}{inf} \underset{v_N \in V_N}{sup} \frac{a_h(u_N , v_N)}{\norm{u_N}_{U_N} \, \norm{v_N}_{V_N}}
		\geq
		\beta.
	\end{equation*}
	\begin{proof}
		Recalling from \eqref{Eq: distributed bilinear form}
		\begin{equation}
			\nonumber
			a_h(u_N , v_N) = 
			\int_{\alpha_{min}}^{\alpha_{max}} \phi(\alpha)
			a_h^\alpha(u_N , v_N) \,\,
			d\alpha,
		\end{equation}
		where, by lemma \eqref{Lem: Frac derive of Poly-fractonomials},
		\begin{align}
			\nonumber
			a_h^\alpha(u_N , v_N) 
			& = 
			a_h^\alpha(\sum_{n=1}^{N} a_n \prescript{(1)}{}{ \mathcal{P}}_{n}^{\mu}(\xi) , \sum_{k=1}^{N} a_k \prescript{(2)}{}{ \mathcal{P}}_{k}^{\mu}(\xi)),
			\quad   \forall \alpha \in (\alpha_{min} , \alpha_{max}),
			\\
			\nonumber
			& 
			=\sum_{n=1}^{N} \sum_{k=1}^{N} a_n a_k C_{n,k}^{\mu,\alpha}
			\int_{-1}^{1}
			(1 - \xi)^{\tilde{\mu}}
			(1 + \xi)^{-\tilde{\mu}} 
			(1 + \xi)^{2\tilde{\mu}} 
			P^{\tilde{\mu} , -\tilde{\mu}}_{k-1}(\xi) \,\,
			P^{-\tilde{\mu} , \tilde{\mu}}_{n-1}(\xi) \,\,
			d\xi, 
		\end{align}
		in which, $\tilde{\mu} = \mu - \alpha/2$, $C_{n,k}^{\mu,\alpha} = \frac{\Gamma(n +\mu)}{\Gamma(n + \tilde{\mu})} \frac{\Gamma(k + \mu)}{\Gamma(k + \tilde{\mu})}$ and $(1 + \xi)^{\tilde{\mu}}$ is replaced by $(1 + \xi)^{-\tilde{\mu}}(1 + \xi)^{2\tilde{\mu}}$. We let $\tilde \mu > -1/2 $, hence the function $(1 + \xi)^{2\tilde{\mu}}$ is nonnegative, nondecreasing, continuous and integrable in the integration domain. Therefore, 
		\begin{equation*}
			a_h^\alpha(u_N , v_N)
			\geq
			\sum_{n=1}^{N} \sum_{k=1}^{N} a_n a_k C_{n,k}^{\mu,\alpha} C_0
			\int_{-1}^{1}
			(1 - \xi)^{\tilde{\mu}}
			(1 + \xi)^{-\tilde{\mu}} 
			P^{\tilde{\mu} , -\tilde{\mu}}_{k-1}(\xi) 
			P^{-\tilde{\mu} , \tilde{\mu}}_{n-1}(\xi) \,\,
			d\xi.
		\end{equation*}
		Moreover, the Jacobi polynomial $P^{-\tilde{\mu} , \tilde{\mu}}_{n-1}(\xi)$ can be expanded as:
		\begin{equation*}
			P^{-\tilde{\mu} , \tilde{\mu}}_{n-1}(\xi) = 
			\sum_{j=0}^{n-1} 
			\binom{n-1+j}{j} \binom{n-1+\tilde{\mu}}{n-1-j} (-1)^{j-n+1} (\frac{1}{2})^j (1 + \xi)^j.
		\end{equation*}
		By multiplying and dividing each term within the summation by $\binom{n-1-\tilde{\mu}}{n-1-j}$ we get
		\begin{equation*}
			P^{-\tilde{\mu} , \tilde{\mu}}_{n-1}(\xi) = 
			\sum_{j=0}^{n-1} 
			\mathcal{A}_{n,j}^{\tilde{\mu}} \binom{n-1+j}{j} \binom{n-1-\tilde{\mu}}{n-1-j} (-1)^{j-n+1} (\frac{1}{2})^j (1 + \xi)^j,
		\end{equation*}
		where $\mathcal{A}_{n,j}^{\tilde{\mu}} = \binom{n-1+\tilde{\mu}}{n-1-j} / \binom{n-1-\tilde{\mu}}{n-1-j} $ is nondecreasing, positive and bounded $\forall n, j, \tilde{\mu}$. Therefore, there exists $C_1 = C_1(n) > 0$ such that
		\begin{align*}
			a_h^\alpha(u_N , v_N)
			& \geq
			\sum_{n=1}^{N} \sum_{k=1}^{N} a_n a_k \, C_{n,k}^{\mu,\alpha} \, C_0 \, C_1(n)
			\int_{-1}^{1}
			(1 - \xi)^{\tilde{\mu}}
			(1 + \xi)^{-\tilde{\mu}} 
			P^{\tilde{\mu} , -\tilde{\mu}}_{k-1}(\xi) 
			P^{\tilde{\mu} , -\tilde{\mu}}_{n-1}(\xi) \,\,
			d\xi
			\\
			& \geq
			C_0
			\sum_{k=1}^{N} a_k^2 \,\, C_{k,k}^{\mu,\alpha} \, C_1(n) \, \epsilon_{k-1}^{\tilde{\mu} , -\tilde{\mu}},
		\end{align*}
		in which $\epsilon_{k-1}^{\tilde{\mu} , -\tilde{\mu}} = \frac{2}{2k-1} \frac{\Gamma (k + \mu) \, \Gamma (k - \mu)}{(k-1)! \, \Gamma (k)}$. Hence,
		\begin{equation}
			\label{Eq: dis. stability num}
			a_h(u_N , v_N)
			\geq
			C_0
			\int_{\alpha_{min}}^{\alpha_{max}} \phi(\alpha) \,\,
			\sum_{k=1}^{N} a_k^2 \, C_{k,k}^{\mu,\alpha} \, C_1(n) \, \epsilon_{k-1}^{\tilde{\mu} , -\tilde{\mu}} \,\,
			d\alpha.
		\end{equation}

		Moreover, we have
		\begin{equation*}
			\norm{v_N}_{V_N}^2 = 
			\int_{\alpha_{min}}^{\alpha_{max}} \phi(\alpha) \,\,
			\norm{\prescript{}{-1}{\mathcal{D}}_{x}^{\alpha/2} v_N}_{L^2(\Omega)}^{2}
			d\alpha,
		\end{equation*}
		where by considering $v_n = \sum_{k=1}^{N} a_k \prescript{(2)}{}{ \mathcal{P}}_{k}^{\mu}(\xi) $, we can write  $\forall \alpha \in (\alpha_{min} , \alpha_{max})$,
		\begin{eqnarray}
			\nonumber
			\norm{\prescript{}{-1}{\mathcal{D}}_{x}^{\alpha/2} v_N}_{L^2(\Omega)}^{2}
			&&=
			\int_{-1}^{1} 
			(\sum_{k=1}^{N} a_k \,\, 
			\frac{\Gamma(k + \mu)}{\Gamma(k + \tilde{\mu})} 
			(1 - \xi)^{\tilde{\mu}} 
			P^{\tilde{\mu} , -\tilde{\mu}}_{k-1}(\xi)  )^2 
			d\xi,
			\\
			\nonumber
			&&\leq
			\int_{-1}^{1} 
			\sum_{k=1}^{N} a_k^2 \,\, 
			\left( \frac{\Gamma(k + \mu)}{\Gamma(k + \tilde{\mu})} \right)^2
			(1 - \xi)^{2\tilde{\mu}} 
			( P^{\tilde{\mu} , -\tilde{\mu}}_{k-1}(\xi) )^2
			d\xi,
			\\
			\nonumber
			&& \text{(By Jensen Inequality)}.
		\end{eqnarray}
		By multiplying the integrand by $(1 + \xi)^{-\tilde{\mu}} (1 + \xi)^{\tilde{\mu}}$  and changing the order of summation and integration, we obtain
		\begin{eqnarray}
			\nonumber
			\norm{\prescript{}{-1}{\mathcal{D}}_{x}^{\alpha/2} v_N}_{L^2(\Omega)}^{2}
			%
			&\leq&
			\sum_{k=1}^{N} a_k^2 \,\, 
			\left( \frac{\Gamma(k + \mu)}{\Gamma(k + \tilde{\mu})} \right)^2
			\int_{-1}^{1}
			(1 - \xi)^{\tilde{\mu}} (1 + \xi)^{-\tilde{\mu}} (1 - \xi^2)^{\tilde{\mu}}
			P^{\tilde{\mu} , -\tilde{\mu}}_{k-1}(\xi) \,\,
			P^{\tilde{\mu} , -\tilde{\mu}}_{k-1}(\xi) \,\,
			d\xi,
			\\
			\nonumber
			&\leq&
			\sum_{k=1}^{N} a_k^2 \,\, 
			\left( \frac{\Gamma(k + \mu)}{\Gamma(k + \tilde{\mu})} \right)^2
			C_2
			\int_{-1}^{1}
			(1 - \xi)^{\tilde{\mu}} (1 + \xi)^{-\tilde{\mu}} 
			P^{\tilde{\mu} , -\tilde{\mu}}_{k-1}(\xi) \,\,
			P^{\tilde{\mu} , -\tilde{\mu}}_{k-1}(\xi) \,\,
			d\xi,
		\end{eqnarray}
		since $\tilde \mu > -1/2$ and consequently $(1 - \xi^2)^{\tilde{\mu}}$ is a nonnegative and integrable in the domain of integration. By the orthogonality of Jacobi polynomials, we get
		\begin{equation*}
			\norm{\prescript{}{-1}{\mathcal{D}}_{x}^{\alpha/2} v_N}_{L^2(\Omega)}^{2}
			\leq
			C_2
			\sum_{k=1}^{N} a_k^2 \,\, 
			\left( \frac{\Gamma(k + \mu)}{\Gamma(k + \tilde{\mu})} \right)^2
			\epsilon_{k-1}^{\tilde{\mu} , -\tilde{\mu}},
		\end{equation*}
		and thus
		\begin{equation}
			\label{Eq: dis. Stability den V}
			\norm{v_N}_{V_N} 
			\leq
			\sqrt{
				C_2
				\int_{\alpha_{min}}^{\alpha_{max}} \phi(\alpha) \,\,
				\sum_{k=1}^{N} a_k^2 \,\, 
				\left( \frac{\Gamma(k + \mu)}{\Gamma(k + \tilde{\mu})} \right)^2
				\epsilon_{k-1}^{\tilde{\mu} , -\tilde{\mu}} \,\,
				d\alpha
			}.
		\end{equation}
		Similarly for $\norm{u_N}_{U_N}^2$:
		\begin{equation*}
			\norm{u_N}_{U_N}^2 = 
			\int_{\alpha_{min}}^{\alpha_{max}} \phi(\alpha) \,\,
			\norm{\prescript{}{-1}{\mathcal{D}}_{x}^{\alpha/2} u_N}_{L^2(\Omega)}^{2}
			d\alpha,
		\end{equation*}
		where $\forall \alpha \in (\alpha_{min} , \alpha_{max})$:
		\begin{eqnarray}
			\nonumber
			\norm{\prescript{}{-1}{\mathcal{D}}_{x}^{\alpha/2} u_N}_{L^2(\Omega)}^{2}
			&&=
			\int_{-1}^{1} 
			(\sum_{n=1}^{N} a_n \,\, 
			\frac{\Gamma(n + \mu)}{\Gamma(n + \tilde{\mu})} 
			(1 + \xi)^{\tilde{\mu}} 
			P^{-\tilde{\mu} , \tilde{\mu}}_{n-1}(\xi)  )^2 
			d\xi,
			\\
			\nonumber
			&&\leq
			\int_{-1}^{1} 
			\sum_{n=1}^{N} a_n^2 \,\, 
			\left( \frac{\Gamma(n + \mu)}{\Gamma(n + \tilde{\mu})} \right)^2
			(1 + \xi)^{2\tilde{\mu}} 
			( P^{-\tilde{\mu} , \tilde{\mu}}_{n-1}(\xi) )^2
			d\xi,
			\\
			\nonumber
			&& \text{By Jensen Inequality}.
		\end{eqnarray}
		Following similar steps, and by multiplying the integrand by $(1 - \xi)^{-\tilde{\mu}} (1 - \xi)^{\tilde{\mu}}$  and changing the order of summation and integration, we obtain  
		\begin{eqnarray}
			\nonumber
			&&
			\norm{\prescript{}{-1}{\mathcal{D}}_{x}^{\alpha/2} u_N}_{L^2(\Omega)}^{2},
			\\
			\nonumber
			&\leq&
			\sum_{n=1}^{N} a_n^2 \,\, 
			\left( \frac{\Gamma(n + \mu)}{\Gamma(n + \tilde{\mu})} \right)^2
			\int_{-1}^{1}
			(1 - \xi)^{-\tilde{\mu}} (1 + \xi)^{\tilde{\mu}} (1 - \xi^2)^{\tilde{\mu}}
			P^{-\tilde{\mu} , \tilde{\mu}}_{n-1}(\xi) \,\,
			P^{-\tilde{\mu} , \tilde{\mu}}_{n-1}(\xi) \,\,
			d\xi,
			\\
			\nonumber
			&\leq&
			\sum_{n=1}^{N} a_n^2 \,\, 
			\left( \frac{\Gamma(n + \mu)}{\Gamma(n + \tilde{\mu})} \right)^2
			C_3
			\int_{-1}^{1}
			(1 - \xi)^{-\tilde{\mu}} (1 + \xi)^{\tilde{\mu}} 
			P^{-\tilde{\mu} , \tilde{\mu}}_{n-1}(\xi) \,\,
			P^{-\tilde{\mu} , \tilde{\mu}}_{n-1}(\xi) \,\,
			d\xi,
		\end{eqnarray}
		since $(1 - \xi^2)^{\tilde{\mu}}$ is a nonnegative and integrable in the domain of integral. Next, by the orthogonality of Jacobi polynomials,
		\begin{equation*}
			\norm{\prescript{}{-1}{\mathcal{D}}_{x}^{\alpha/2} u_N}_{L^2(\Omega)}^{2}
			\leq
			C_3
			\sum_{n=1}^{N} a_n^2 \,\, 
			\left( \frac{\Gamma(n + \mu)}{\Gamma(n + \tilde{\mu})} \right)^2
			\epsilon_{n-1}^{-\tilde{\mu} , \tilde{\mu}}.
		\end{equation*}
		Therefore,
		\begin{equation}
			\label{Eq: dis. Stability den U}
			\norm{u_N}_{U_N} 
			\leq
			\sqrt{
				C_3
				\int_{\alpha_{min}}^{\alpha_{max}} \phi(\alpha) \,\,
				\sum_{n=1}^{N} a_n^2 \,\, 
				\left( \frac{\Gamma(n + \mu)}{\Gamma(n + \tilde{\mu})} \right)^2
				\epsilon_{n-1}^{\tilde{\mu} , -\tilde{\mu}} \,\,
				d\alpha
			},
		\end{equation}
		where $\epsilon_{n-1}^{-\tilde{\mu} , \tilde{\mu}} = \epsilon_{n-1}^{\tilde{\mu} , -\tilde{\mu}}$. 
		
		Therefore, using \eqref{Eq: dis. stability num}, \eqref{Eq: dis. Stability den V} and \eqref{Eq: dis. Stability den U},
		\begin{equation*}
			\underset{u_N \in U_N}{inf} \underset{v_N \in V_N}{sup} \frac{a_h(u_N , v_N)}{\norm{u_N}_{U_N} \, \norm{v_N}_{V_N}}
			\geq
			\frac{C_0 \,\, C_{1_{min}}}{\sqrt{C_2 \,\, C_3}} = \beta.
		\end{equation*}
	\end{proof}
\end{theorem}

%
\subsection{Projection Error Analysis}
\label{Sec:Projection Error analysis} 
%

In this section, we investigate the error due to the projection of the true solution onto the defined set of basis functions.
\vspace{0.2 cm}
\begin{theorem}
	\label{Thm: Error analysis}
	Let $\frac{d^r u}{dt^r} \in U$, that is, $\int_{\alpha_{min}}^{\alpha{max}} \phi(\alpha) \, \norm{\prescript{}{0}{\mathcal{D}}_{t}^{r + \alpha/2} u \,}^2_{L^2(\Omega)} \, d\alpha < \infty$ and $u_N$ denotes the projection of the exact solution $u$. Then,
	\begin{equation}
		\label{Eq: projection error}
		\norm{u - u_N}^{2}_{U} \leq \mathcal{C} \, N^{-2r} \, \int_{\alpha_{min}}^{\alpha_{max}}  \phi(\alpha) \norm{\prescript{}{0}{\mathcal{D}}_{t}^{r + \alpha/2} u \,}^2_{L^2(\Omega)} d\alpha.
	\end{equation}
	\begin{proof}
		By performing an affine mapping from $t\in[0,T]$ to the standard domain $\xi \in [-1,1]$, we expand the exact solution $u$ in terms of the following infinite series of Jacobi poly-fractonomials
		\begin{equation}
			\label{Eq: exact solution expansion}
			u(\xi) = \sum_{n=1}^{\infty} c_n  \prescript{(1)}{}{ \mathcal{P}}_{n}^{\mu}(\xi).
		\end{equation}
		Then, we note that by using \eqref{Lem: Frac derive of Poly-fractonomials} and \eqref{Eq: RL in xL-xR},
		\begin{align*}
			\prescript{}{0}{\mathcal{D}}_{t}^{r + \alpha/2} \,\,  u(\xi(t))
			& = (\frac{2}{T})^{r+\alpha/2} \, \prescript{}{-1}{\mathcal{I}}_{\xi}^{\mu - \alpha/2} \, \frac{d^r}{d\xi^r}  \prescript{}{-1}{\mathcal{D}}_{\xi}^{\mu} \,\, u(\xi),
			\\
			&=
			(\frac{2}{T})^{r+\alpha/2} \, \sum_{n=1}^{\infty} c_n  \left(\frac{\Gamma(n + \mu)}{\Gamma(n)}\right) 
			\prescript{}{-1}{\mathcal{I}}_{\xi}^{\mu - \alpha/2} \,\, \frac{d^r}{d\xi^r}[P_{n-1}(\xi)],
		\end{align*}
		where, 
		\begin{align*}
			\frac{d^r}{d\xi^r}[P_{n-1}(\xi)] =
			\left\{
			\begin{array} {ll}
				\frac{(n-1+r)!}{2r(n-1)!} P_{n-1-r}^{r,r} (\xi) & r<n,
				\\ 
				0 & r\geq n.
			\end{array}
			\right.
		\end{align*}
		Thus, by multiplying with a proper weight function, $w(\xi) = (1 + \xi)^{r/2-\mu+\alpha/2} \,\, (1 - \xi)^{r/2}$, the right-hand-side of \eqref{Eq: projection error} takes the form
		\begin{align}
			\label{Eq: projection error - R1}
			\int_{\alpha_{min}}^{\alpha_{max}}  \phi(\alpha) \, (\frac{2}{T})^{2r+\alpha/2} \,
			&\bigg|\bigg|
			(1 + \xi)^{r/2-\mu+\alpha/2} \,\, (1 - \xi)^{r/2} \,\,
			\\
			\nonumber
			&
			\sum_{n=r+1}^{\infty} c_n  \left(\frac{\Gamma(n + \mu)}{\Gamma(n)}\right) \frac{(n-1+r)!}{2 r (n-1)!} \,\,
			\prescript{}{-1}{\mathcal{I}}_{\xi}^{\mu - \alpha/2} \,\, 
			P_{n-1-r}^{r,r}(\xi)
			\bigg|\bigg| ^2_{L^2(\Omega)}\,
			d\alpha.
		\end{align}
		By expanding the Jacobi polynomial as 
		\begin{align*}
			P_{n-1-r}^{r,r}(\xi) = (-1)^{n-1-r} \sum_{j=0}^{n-1-r} \, \binom{n-1+r+j}{j} \binom{n-1}{n-1-r-j} (\frac{-1}{2})^j \, (1 + \xi)^j,
		\end{align*}
		and changing the order of summation and the integration, we obtain the fractional integral as
		\begin{align}
			\label{Eq: projection error - R3}
			&\prescript{}{-1}{\mathcal{I}}_{\xi}^{\mu - \alpha/2} \,\, P_{n-1-r}^{r,r}(\xi),
			\\
			\nonumber
			& =
			(-1)^{n-1-r} \, (1+\xi)^{\mu-\alpha/2}
			\sum_{j=0}^{n-1-r} \, 
			\binom{n-1+r+j}{j} \binom{n-1}{n-1-r-j} (\frac{-1}{2})^j \, 
			\frac{\Gamma(1+j)}{\Gamma(1+j+\mu-\alpha/2)} \, (1+\xi)^{j},
			\\
			\nonumber
			& = 
			(-1)^{n-1-r} \, (1+\xi)^{\mu-\alpha/2} \,\,
			\sum_{q=0}^{n-1-r} \, \tilde c_q(\alpha) \,\, P_{q}^{r,r} (\xi),
		\end{align}
		where, the coefficient, $\tilde c_q(\alpha)$, can be obtained using the orthogonality of Jacobi polynomials. Hence, by taking $C_n =  c_n \, (-1)^{n-1} \, \frac{\Gamma(n + \mu)}{\Gamma(n)}$, \eqref{Eq: projection error - R1} takes the form
		\begin{align}
			\label{Eq: projection error - R4}
			&
			\int_{\alpha_{min}}^{\alpha_{max}}  \phi(\alpha) \, (\frac{2}{T})^{2r+\alpha} \, 
			\norm{
				(1 + \xi)^{r/2} \,\, (1 - \xi)^{r/2} \,\,
				\sum_{n=r+1}^{\infty}  \,  C_n \, \frac{(n-1+r)!}{2 r (n-1)!} \,\,
				(-1)^{-r} \, 
				\sum_{q=0}^{n-1-r} \, \tilde c_q(\alpha) \, P_{q}^{r,r}(\xi)
			}^2_{L^2(\Omega)}\,
			d\alpha,
			\\
			\nonumber
			= &
			\int_{\alpha_{min}}^{\alpha_{max}}  \phi(\alpha) \, (\frac{2}{T})^{2r+\alpha} \, 
			\left[
			\int_{-1}^{1}
			(1 + \xi)^{r} \,\, (1 - \xi)^{r} \,\, 
			\left(
			\sum_{n=r+1}^{\infty} \, C_n \, \frac{(n-1+r)!}{2 r (n-1)!} \,\,
			\sum_{q=0}^{n-1-r} \, \tilde c_q(\alpha) \, P_{q}^{r,r}(\xi)
			\right)^2
			d\xi\,
			\right]
			d\alpha.
		\end{align}
		Then, we change the order of two summations in order to use the orthogonality of Jacobi polynomials and obtain
		\begin{align}
			\label{Eq: projection error - R5}
			&
			\int_{\alpha_{min}}^{\alpha_{max}}  \phi(\alpha) \, (\frac{2}{T})^{2r+\alpha} \, 
			\left[
			\int_{-1}^{1}
			(1 + \xi)^{r} \,\, (1 - \xi)^{r} \,\,
			\left(
			\sum_{q=0}^{\infty} \,\,
			\sum_{n=r+1+q}^{\infty} \, C_n \, \frac{(n-1+r)!}{2 r (n-1)!} \,\,
			\, \tilde c_q(\alpha) \, P_{q}^{r,r} (\xi)
			\right)^2
			d\xi\,
			\right]
			d\alpha,
			\\
			\nonumber
			= &
			\int_{\alpha_{min}}^{\alpha_{max}}  \phi(\alpha) \, (\frac{2}{T})^{2r+\alpha} \, 
			\left[
			\int_{-1}^{1}
			(1 + \xi)^{r} \,\, (1 - \xi)^{r} \,\,
			\sum_{q=0}^{\infty} \,\,
			\left(\sum_{n=r+1+q}^{\infty} \, C_n \, \frac{(n-1+r)!}{2 r (n-1)!}\right)^2 \,\,
			\tilde c^2_q(\alpha) \, (P_{q}^{r,r} (\xi))^2
			d\xi\,
			\right]
			d\alpha,
			\\
			\nonumber
			= &
			\sum_{q=0}^{\infty} \,\,
			\left(\sum_{n=r+1+q}^{\infty} \, C_n \, \frac{(n-1+r)!}{2 r (n-1)!}\right)^2 \,\,
			\int_{-1}^{1}
			(1 + \xi)^{r} \,\, (1 - \xi)^{r} \,\,
			(P_{q}^{r,r} (\xi))^2
			d\xi\,
			\int_{\alpha_{min}}^{\alpha_{max}}  \phi(\alpha) \, (\frac{2}{T})^{2r+\alpha} \, \tilde c^2_q(\alpha) \,
			d\alpha,
			\\
			\nonumber
			= &
			\sum_{q=0}^{\infty} \,\,
			\left(\sum_{n=r+1+q}^{\infty} \, C_n \, \frac{(n-1+r)!}{2 r (n-1)!}\right)^2 \,\,
			\frac{2^{2r+1}}{2q+2r+1} \frac{((q+1)!)^2}{q!(q+2r)!}
			\,
			\int_{\alpha_{min}}^{\alpha_{max}}  \phi(\alpha) \, (\frac{2}{T})^{2r+\alpha} \, \tilde c^2_q(\alpha) \,
			d\alpha,
			\\
			\nonumber
			= &
			\frac{2^{2r+1} (\frac{2}{T})^{2r}}{(2r)^2} \, 
			\sum_{q=0}^{\infty} \,\,
			\left(\sum_{n=r+1+q}^{\infty} \, C_n \, \frac{(n-1+r)!}{(n-1)!}\right)^2 \,\,
			\frac{(q+1)^2 \, q!}{(2q+2r+1)(q+2r)!}
			\,
			\int_{\alpha_{min}}^{\alpha_{max}}  \phi(\alpha) \, (\frac{2}{T})^{\alpha} \, \tilde c^2_q(\alpha) \,
			d\alpha.
		\end{align}
		Moreover, using the approximation of the solution given in \eqref{Eq: Approx SolN Modal} and by multiplying with the proper weight functions, the left-hand-side of \eqref{Eq: projection error} takes the form
		\begin{align}
			\label{Eq: projection error - L1}
			&\norm{(1+\xi)^{-\mu+\alpha/2} \, (u - u_N)}^{2}_{U} ,
			\\
			\nonumber
			& =
			\int_{\alpha_{min}}^{\alpha_{max}}  \phi(\alpha) \, (\frac{2}{T})^{\alpha} 
			\bigg|\bigg|
			(1+\xi)^{-\mu+\alpha/2} 
			\sum_{n=N+1}^{\infty} c_n \,\, \left(\frac{\Gamma(n + \mu)}{\Gamma(n)}\right) \prescript{}{-1}{\mathcal{I}}_{\xi}^{\mu - \alpha/2}  P_{n-1}(\xi) 
			\bigg|\bigg|
			^2_{L^2(\Omega)} 
			d\alpha,
		\end{align}
		in which, $\prescript{}{-1}{\mathcal{D}}_{\xi}^{\alpha/2} = \prescript{}{-1}{\mathcal{D}}_{\xi}^{\alpha/2 - \mu + \mu} =  \prescript{}{-1}{\mathcal{I}}_{\xi}^{\mu - \alpha/2} \,\, \prescript{}{-1}{\mathcal{D}}_{\xi}^{\mu}$ and the fractional derivative is taken using \eqref{Lem: Frac derive of Poly-fractonomials}. By expanding the Legendre polynomial as
		\begin{align*}
			P_{n-1}(\xi) = (-1)^{n-1} \sum_{j=0}^{n-1} \, \binom{n-1+j}{j} \binom{n-1}{n-1-j} (\frac{-1}{2})^j \, (1 + \xi)^j, \quad
		\end{align*}
		and following similar steps as in \eqref{Eq: projection error - R3}, we obtain the fractional integral as
		\begin{align}
			\label{Eq: projection error - L2}
			\prescript{}{-1}{\mathcal{I}}_{\xi}^{\mu - \alpha/2}  P_{n-1}(\xi)
			= (-1)^{n-1} \, (1 + \xi)^{\mu-\alpha/2} \sum_{q=0}^{n-1} \, \tilde{a}_q(\alpha) \, P_q(\xi),
		\end{align}
		where the coefficient, $\tilde a_q(\alpha)$, can be obtained using the orthogonality of Legendre polynomials. Hence, \eqref{Eq: projection error - L1} takes the form
		\begin{align}
			\label{Eq: projection error - L3}
			\norm{(1+\xi)^{-\mu+\alpha/2} \, (u - u_N)}^{2}_{U}
			=\int_{\alpha_{min}}^{\alpha_{max}}  \phi(\alpha) \, (\frac{2}{T})^{\alpha} \,\,
			\int_{-1}^{1}
			\left(
			\sum_{n=N+1}^{\infty} \,  C_n \,\, 
			\sum_{q=0}^{n-1} \, \tilde{a}_q(\alpha) \, P_q(\xi) \,
			\right)^2
			d\xi \,
			d\alpha,
		\end{align}
		in which, $C_n =  c_n \, (-1)^{n-1} \, \frac{\Gamma(n + \mu)}{\Gamma(n)}$. We change the order of two summations to use the orthogonality of Legendre polynomials and obtain
		\begin{align}
			\label{Eq: projection error - L4}
			&
			\int_{\alpha_{min}}^{\alpha_{max}}  \phi(\alpha) \, (\frac{2}{T})^{\alpha} \,\,
			\int_{-1}^{1}
			\left(
			\sum_{q=0}^{N}  \sum_{n=q+1}^{\infty}  C_n \,\, \tilde{a}_q(\alpha) \, P_q(\xi) \,
			+ \sum_{q=N+1}^{\infty}  \sum_{n=q+1}^{\infty}  C_n \,\, \tilde{a}_q(\alpha) \, P_q(\xi) \,
			\right)^2
			d\xi \,
			d\alpha,
			\\
			\nonumber
			& =
			\int_{\alpha_{min}}^{\alpha_{max}}  \phi(\alpha) \, (\frac{2}{T})^{\alpha} \,\,
			\left(
			\sum_{q=0}^{N} \left( \sum_{n=q+1}^{\infty}  C_n \right)^2 \tilde{a}^2_q(\alpha) \, 
			\int_{-1}^{1} \, (P_q(\xi))^2 \, d\xi
			+ 
			\sum_{q=N+1}^{\infty} \left( \sum_{n=q+1}^{\infty}  C_n \right)^2 \tilde{a}^2_q(\alpha) \, 
			\int_{-1}^{1} \, (P_q(\xi))^2 \, d\xi
			\right) \,
			d\alpha,
			\\
			\nonumber
			& \leq
			\sum_{q=N+1}^{\infty} \left( \sum_{n=q+1}^{\infty}  C_n \right)^2 
			\frac{2}{2q+1} \,
			\int_{\alpha_{min}}^{\alpha_{max}}  \phi(\alpha) \, (\frac{2}{T})^{\alpha} \, \tilde{a}^2_q(\alpha) \, d\alpha,
			\\
			\nonumber
			& \leq
			\left( \frac{N!}{(N-r)!} \right)^2
			\sum_{q=N+1}^{\infty} \left( \sum_{n=q+1}^{\infty}  C_n \, \frac{(n-1+r)!}{(n-1)!} \right)^2 
			\frac{2}{2q+1} \, \frac{(q+1)^2 \, q!}{(2q+2r+1)(q+2r)!} \,
			\int_{\alpha_{min}}^{\alpha_{max}}  \phi(\alpha) \, (\frac{2}{T})^{\alpha} \, \tilde{a}^2_q(\alpha) \, d\alpha,
			\\
			\nonumber
			& \leq
			\left( \frac{N!}{(N-r)!} \right)^2
			\int_{\alpha_{min}}^{\alpha_{max}}  \phi(\alpha) \, (\frac{2}{T})^{2r+\alpha}\, 
			\norm{(1 + \xi)^{r/2-\mu+\alpha/2} \,\, (1 - \xi)^{r/2} \,\, \prescript{}{-1}{\mathcal{D}}_{\xi}^{r + \alpha/2} u \,}^2_{L^2(\Omega)} d\alpha,
			\\
			\nonumber
			& \leq
			\left( \frac{N!}{(N-r)!} \right)^2
			\int_{\alpha_{min}}^{\alpha_{max}}  \phi(\alpha) \, 
			\norm{\prescript{}{0}{\mathcal{D}}_{t}^{r + \alpha/2} u \,}^2_{L^2(\Omega)} d\alpha.
		\end{align}
		Therefore, 
		\begin{align*}
			\norm{u - u_N}^{2}_{U} \leq
			\norm{(1+\xi)^{-\mu+\alpha/2} \, (u - u_N)}^{2}_{U} \leq
			\mathcal{C} \, N^{-2r} \, \int_{\alpha_{min}}^{\alpha_{max}}  \phi(\alpha) \norm{\prescript{}{0}{\mathcal{D}}_{t}^{r + \alpha/2} u \,}^2_{L^2(\Omega)} d\alpha.
		\end{align*}
	\end{proof}
\end{theorem}
\begin{rem}
	Since the inf-sup condition holds (see Theorem \eqref{Thm: Stability of scheme}), by the Banach-Ne$\breve{c}$as-Babu$\breve{s}$ka theorem \cite{ern2013theory}, the error in the numerical scheme is less that or equal to a constant times the projection error. Choosing the projection $u_N$ in Theorem \eqref{Thm: Error analysis}, we infer the spectral accuracy of the scheme.
\end{rem}


%
\section{Fractional Collocation Method: Nodal Expansion}
\label{Sec: Modal: Collocation Method}
%
Next, we represent the solution via the following poly-fractonomial \textit{nodal} expansion as
\begin{equation}
	\label{Eq: Approx SolN Fractional Lagrange}
	u_N(\xi) = \sum_{j=1}^{N} u_N(\xi_j) \, h^{\mu}_j(\xi),
\end{equation}
where $h^{\mu}_j(\xi)$ represent \textit{fractional Lagrange interpolants} FLIs, which are all of fractional order $(N+\mu-1)$ and constructed using the aforementioned interpolations points $-1=\xi_1 < \xi_2<\cdots<\xi_N=1$ as:
\begin{equation}
	\label{Eq: Fractional Lagrange Interpolants}
	h^{\mu}_j(\xi) = \Big( \frac{\xi -x_1}{x_j-x_1} \Big)^{\mu} \prod_{\substack{k=1 \\ k\neq j}}^{N} \Big( \frac{\xi -x_k}{x_j-x_k} \Big),\,\,\,\quad   j=2,3,\cdots,N.
\end{equation}
Because of the homogeneous Dirichlet boundary condition(s) in \eqref{Eq: distributional ICs-1} and \eqref{Eq: distributional ICs-2}, $u_N(-1)= 0$, and thus we only construct $h^{\mu}_j(\xi)$ for $j=2,3,\cdots,N$. We note that FLIs satisfy the Kronecker delta property, i.e., $h^{\mu}_j(\xi_k) = \delta_{jk}$, at interpolation points, however they vary as a poly-fractonomial between $\xi_k$'s.  

%
\subsection{Differentiation Matrices $\textbf{D}^{\alpha}$ and $\textbf{D}^{1+\alpha}$, $\alpha \in (0,1)$}
\label{Sec: Diff Mat}
%
By breaking the domain of integration in $\alpha$, \eqref{Eq: Distributed derivative-I} takes the form
\begin{equation}
	\label{Eq: Distributed derivative for coll - I}
	\int_{\alpha_{min}}^{1} \phi(\alpha) \prescript{}{0}{\mathcal{D}}_{t}^{\alpha} u(t) d\alpha
	+ \int_{1}^{\alpha_{max}} \phi(\alpha) \prescript{}{0}{\mathcal{D}}_{t}^{\alpha} u(t) d\alpha
	= f(t; u), \quad \forall t \in (0,T].
\end{equation}
Following \cite{Zayernouri14-SIAM-Collocation}, we obtain the corresponding fractional differentiation matrices $\textbf{D}^{\alpha}$ and $\textbf{D}^{1+\alpha}$, $\alpha \in (0,1)$ by substituting \eqref{Eq: Fractional Lagrange Interpolants} in \eqref{Eq: Approx SolN Fractional Lagrange} and taking the $\alpha$-th order fractional derivative. These matrices are given as:
\begin{eqnarray}
	\label{Eq: Diff Matrix Sigma Def}
	\textbf{D}^{\alpha}_{ij} = \frac{1}{(\xi_j+1)^{\mu}} \, \sum_{n=1}^{N}\beta^{j}_n \,
	\sum_{q=  \lceil \alpha - \mu \rceil }^{n-1}\, b_{nq} \,  (\xi_i+1)^{q+\mu-\alpha},
\end{eqnarray}
and
\begin{eqnarray}
	\label{Eq: Diff Matrix Sigma plus one Def}
	\textbf{D}^{1+\alpha}_{ij} = \frac{1}{(\xi_j+1)^{\mu}}\Big[ \, \sum_{n=1}^{N}\beta^{j}_n \,
	\sum_{q=  \lceil \alpha - \mu \rceil }^{n-1}\,b_{nq}(q+\mu-\alpha)\, (\xi_i+1)^{q+\mu-\alpha-1}\Big],
\end{eqnarray}
in which $ \lceil \alpha-\mu \rceil $ denotes the ceiling of $\alpha-\mu$ and
\begin{eqnarray}
	\label{Eq: bnq}
	b_{nq} = (-1)^{n+q-1} (\frac{1}{2})^{q} 
	\begin{pmatrix} n-1+ q \\ q \end{pmatrix}
	\begin{pmatrix} n-1+\mu  \\ n-1-q \end{pmatrix}
	\frac{\Gamma(q+\mu+1)}{\Gamma(q+\mu-\alpha +1)}.
\end{eqnarray}
The coefficients, $\beta^{j}_n$, are the coefficients in expansion of the polynomial $p_j(\xi) =\prod_{\substack{k=1 \\ k\neq j}}^{N} \Big( \frac{\xi -\xi_k}{\xi_j -\xi_k} \Big)$ in terms of Jacobi polynomials as
\begin{equation}
	\label{Eq: Poly representation}
	\prod_{\substack{k=1 \\ k\neq j}}^{N} \Big( \frac{\xi -\xi_k}{x_j-x_k} \Big) = \sum_{n=1}^{N} \beta^{j}_n P^{-\mu,\mu}_{n-1}(\xi).   
\end{equation}
Due to the orthogonality of the Jacobi poly-fractonomials $P^{-\mu,\mu}_{n-1}(\xi)$ with respect to the weight function $w(\xi) = (1-\xi)^{-\mu}(1+\xi)^{-\mu}$, these coefficients can be computed efficiently only \textit{once} by employing a proper Guass-Lobatto-Jacobi quadrature rule.

Therefore, by substituting the nodal expansion \eqref{Eq: Fractional Lagrange Interpolants} into \eqref{Eq: Distributed derivative for coll - I}, performing an affine mapping from $[\alpha_{min},\alpha_{max}]$ to the standard domain $[-1,1]$, and employing a proper quadrature rule in $\alpha$-domain, \eqref{Eq: Distributed derivative for coll - I} can be written as
\begin{align}
	\label{Eq: Distributed derivative for coll - II}
	&
	\sum_{j=2}^{N}
	\left[
	\sum_{q}^{Q} 
	w_q \, \phi(\alpha_q) \, (\frac{2}{T})^{\alpha_q} \, 
	\left( \textbf{D}^{\alpha_q}_{ij} + \textbf{D}^{1+\alpha_q}_{ij} \right)
	\right] u_N(\xi_j)
	= f(\xi_i),
	\\
	\nonumber
	&
	\sum_{j=2}^{N} \,\, 
	\mathscr{D}_{i,j} \,\, u_N(\xi_j)
	= f(\xi_i),  \quad i=2,3,\cdots,N.
\end{align}
\begin{rem}
	Multi-term problems can be generalized to the distributed order counterparts through the definition of distribution function $\phi(\alpha)$. For instance, if the operator consists of multiple fractional orders $\,\, 0 < \alpha_1 < \alpha_2 < \cdots < \alpha_P \leq 2$, the corresponding multi-term problem
	\begin{align*}
		\sum_{p=1}^{p=P} \,\, \prescript{}{0}{\mathcal{D}}_{t}^{\alpha_p} u(t) = f(t)
	\end{align*}
	can be represented as a distributed order problem of the form \eqref{Eq: Distributed derivative-I}, in which $\phi(\alpha) = \sum_{p=1}^{p=P} \,\, \delta(\alpha-\alpha_p)$. We note that in this case, the distributed fractional Sobolev space, $\prescript{\phi}{}{\mathcal H}(\mathbb{R})$, coincides with the fractional Sobolev space ,$H^{\alpha_P}(\mathbb{R})$. 
	The choice of collocation/interpolation points is the key to construct well-conditioned linear systems with optimal approximability. In the present work, we leave $\mu$ in expansion \eqref{Eq: Approx SolN Fractional Lagrange} as a free interpolation parameter to capture possible singularities and employ the zeros of Legendre polynomials as the interpolation collocation/interpolation points. 
	
	%
\end{rem}
%

%
\section{Numerical Simulations}
\label{Sec: Numerical Simulation} 
%

In order to examine the convergence of the schemes with \textit{modal} and \textit{nodal} expansions, we consider problems with smooth and non-smooth solutions.

%
\subsection{Smooth Solutions}
\label{Sec: Numerical Simulation - smoosth} 
%

Let $\alpha \in [0,2]$ and consider the following two cases:
\begin{itemize}
	\item	Case I: $u^{ext} = t^5$, \, $\phi(\alpha) = \Gamma(6-\alpha)/5!$  , \, $f(t) =\frac{(t^5-t^3)}{\log(t)}$
	
	\item 	Case II: $u^{ext} = t^3$, \, $\phi(\alpha) = \Gamma(4-\alpha) \sinh(\alpha)$, \, $f(t) = \frac{6 t ( t^2 - \cosh(2) -\sinh(2) \log(t) )}{(\log(t)^2 -1)}$.
\end{itemize}
By taking the simulation time $T=2$ and for different choices of $\mu$, we provide the convergence study in $L^{\infty}$-norm, $L^{2}$-norm, $H^{\mu_1}$-norm and $\prescript{\phi}{}{\mathcal H}$-norm using the PG scheme and in $L^{\infty}$-norm using the collocation scheme. It is observed that the choice of $\mu$ has an important effect on the convergence behaviour of the scheme. For instance, since the exact solution is a polynomial, as $\mu \rightarrow 1$, we recover the exponential convergence in capturing the exact solution. 



\begin{table}[h]
	\center
	\caption{\label{Table: PG Distributed-Order Results 1} Case-I; PG scheme convergence study in $L^{\infty}$-norm, $L^{2}$-norm, $H^{\mu_1}$-norm and $\prescript{\phi}{}{\mathcal H}$-norm, where $T=2$.}
	\vspace{-0.2 in}
	%
	%
	\scalebox{0.75}{
		\begin{tabular}
			%
			{r c c c c c}
			&& $\mu_1 = 1+10^{-4}$  &&  \\
			\hline
			\hline
			N & $L^{\infty}$-Error &  $L^{2}$-Error  &  $H^{\mu_1}$-Error & $\prescript{\phi}{}{\mathcal H}$-Error \\[4pt]
			\hline
			2 &  $9.49784$ &  $3.38063$ &  $20.604$ & $ 6.52507 $ \\
			\hline
			4 &  $0.163486$ &  $0.0823368$ &  $0.802757$ & $ 0.187176 $ \\
			\hline
			6 &  $9.71043\times 10^{-8}$ &  $6.7433 \times 10^{-8}$ &  $8.37613 \times 10^{-7}$ & $ 1.70551 \times 10^{-7} $ \\
			\hline
			8 &  $2.9053 \times 10^{-9}$ &  $2.32457 \times 10^{-9}$ &  $ 3.53574\times 10^{-8}$ & $ 6.59486 \times 10^{-9} $ \\
			\hline
			10 &  $2.27748 \times 10^{-10}$ &  $2.01002 \times 10^{-10}$ &  $3.67074 \times 10^{-9}$ & $ 6.38469 \times 10^{-10} $ \\
			\hline
			%
		\end{tabular}
	}
	\\
	\scalebox{0.8}{
		\begin{tabular}
			%
			{r c c c c c}
			&& $\mu_1 = 1.1$  &&  \\
			\hline
			\hline
			N & $L^{\infty}$-Error &  $L^{2}$-Error  &  $H^{\mu_1}$-Error & $\prescript{\phi}{}{\mathcal H}$-Error \\
			\hline
			2 &  $9.6776$ &  $ 3.2898$ &  $ 23.3004$ & $ 6.38693 $ \\
			\hline
			4 &  $0.160434$ &  $0.0661304$ &  $ 0.872809$ & $ 0.157957 $ \\
			\hline
			6 &  $0.0000947942$ &  $ 0.0000589784$ &  $0.00107458$ & $ 0.00015822 $ \\
			\hline
			8 &  $3.10668 \times 10^{-6}$ &  $2.19939 \times 10^{-6}$ &  $0.0000507737$ & $ 6.59429 \times 10^{-6} $ \\
			\hline
			10 &  $2.48519 \times 10^{-7}$ &  $1.9822 \times 10^{-7}$ &  $5.5753 \times 10^{-6}$ & $ 6.61409 \times 10^{-7} $ \\
			\hline
			%
		\end{tabular}
	}
	\\
	\scalebox{0.8}{
		\begin{tabular}
			%
			{r c c c c c}
			&& $\mu_1 = 1.5$  &&  \\
			\hline
			\hline
			N & $L^{\infty}$-Error &  $L^{2}$-Error  &  $H^{\mu_1}$-Error & $\prescript{\phi}{}{\mathcal H}$-Error \\
			\hline
			2 &  $9.8476$ &  $ 3.10681$ &  $ 35.8457$ & $ 5.96161 $ \\
			\hline
			4 &  $0.102534$ &  $0.0264974 $ &  $ 0.949183$ & $ 0.0718496 $ \\
			\hline
			6 &  $0.000584995$ &  $ 0.00015106$ &  $0.0117235 $ & $ 0.000524729 $ \\
			\hline
			8 &  $0.0000272655$ &  $7.37649 \times 10^{-6}$ &  $0.000989158 $ & $ 0.0000306404 $ \\
			\hline
			10 &  $2.75271 \times 10^{-6}$ &  $7.75346 \times 10^{-7}$ &  $0.000158823$ & $ 3.72512 \times 10^{-6} $ \\
			\hline									
			%
		\end{tabular}
	}
\end{table}
\begin{table}[h]
	\center
	\caption{\label{Table: PG Distributed-Order Results 2} Case-II; PG scheme convergence study in $L^{\infty}$-norm, $L^{2}$-norm, $H^{\mu_1}$-norm and $\prescript{\phi}{}{\mathcal H}$-norm, where $T=2$.}
	\vspace{-0.2 in}
	%
	%
	\scalebox{0.75}{
		\begin{tabular}
			%
			{r c c c c c}
			&& $\mu_1 = 1 + 10^{-4}$  &&  \\
			\hline
			\hline
			N & $L^{\infty}$-Error &  $L^{2}$-Error  &  $H^{\mu_1}$-Error & $\prescript{\phi}{}{\mathcal H}$-Error \\
			\hline
			2 &  $0.379134$ &  $0.325253$ &  $1.44392$ & $ 1.86897 $ \\
			\hline
			4 &  $6.80222 \times 10^{-7}$ &  $6.33141 \times 10^{-7}$ &  $4.61395 \times 10^{-6}$ & $5.29606 \times 10^{-6} $ \\
			\hline
			6 &  $5.22608\times 10^{-8}$ &  $4.52071 \times 10^{-8}$ &  $4.80236 \times 10^{-7}$ & $5.08899 \times 10^{-7} $ \\
			\hline
			8 &  $1.27547 \times 10^{-8}$ &  $9.98313 \times 10^{-9}$ &  $1.0532 \times 10^{-7}$ & $1.049 \times 10^{-7} $ \\
			\hline
			10 &  $7.31142 \times 10^{-9}$ &  $7.21402 \times 10^{-9}$ &  $3.44882 \times 10^{-8}$ & $3.39574 \times 10^{-8} $ \\
			\hline
			%
		\end{tabular}
	}
	\\
	\scalebox{0.8}{
		\begin{tabular}
			%
			{r c c c c c}
			&& $\mu_1 = 1.1$  &&  \\
			\hline
			\hline
			N & $L^{\infty}$-Error &  $L^{2}$-Error  &  $H^{\mu_1}$-Error & $\prescript{\phi}{}{\mathcal H}$-Error \\
			\hline
			2 &  $0.369682$ &  $0.263829$ &  $1.45384$ & $ 1.62458 $ \\
			\hline
			4 &  $0.000646557$ &  $0.000569995$ &  $ 0.00548608$ & $ 0.00499413 $ \\
			\hline
			6 &  $0.0000458334$ &  $0.0000438926$ &  $0.000636023$ & $ 0.000511403 $ \\
			\hline
			8 &  $7.74333 \times 10^{-6}$ &  $7.36329 \times 10^{-6}$ &  $0.000147177$ & $ 0.000107932 $ \\
			\hline
			10 &  $2.02013 \times 10^{-6}$ &  $1.84714 \times 10^{-6}$ &  $0.000048212$ & $ 0.0000327428 $ \\
			\hline
			%
		\end{tabular}
	}
	\\
	\scalebox{0.8}{
		\begin{tabular}
			%
			{r c c c c c}
			&& $\mu_1 = 1.5$  &&  \\
			\hline
			\hline
			N & $L^{\infty}$-Error &  $L^{2}$-Error  &  $H^{\mu_1}$-Error & $\prescript{\phi}{}{\mathcal H}$-Error \\
			\hline
			2 &  $0.288508$ &  $0.114871$ &  $1.25471$ & $ 0.848595 $ \\
			\hline
			4 &  $0.00403916$ &  $0.00163979$ &  $0.0511667$ & $ 0.0190804 $ \\
			\hline
			6 &  $0.000406095$ &  $0.000169817$ &  $0.0106909$ & $ 0.00268063 $ \\
			\hline
			8 &  $0.0000789352$ &  $0.0000336939$ &  $0.00358698$ & $ 0.000671243 $ \\
			\hline
			10 &  $0.0000219275$ &  $9.49574 \times 10^{-6}$ &  $0.00153771$ & $ 0.000228446 $ \\
			\hline
			%
		\end{tabular}
	}

\end{table}
\begin{table}[h]
	\center
	\caption{\label{Table: col. Distributed-Order Results 1} Case-I and II; collocation scheme convergence study in $L^{\infty}$-norm, where $T=2$.}
	\vspace{-0.1in}
	\scalebox{0.7}{
		\begin{tabular}
			{ l  c c  c c  c c }
			\hline
			& \multicolumn{2}{c}{$\mu = 1 - 10^{-10}$} & \multicolumn{2}{c}{$\mu = 7/10$} & \multicolumn{2}{c}{$\mu = 1/10$}  \\
			&&&&&& \\[-2.5ex]
			\hline
			N & Case-I & Case-II  & Case-I & Case-II & Case-I & Case-II   \\
			\hline			
			\hline
			2 &  $ 2.59\times 10^{+1}$ & $ 5.74$   &  $ 3.0 \times 10^{+1}$ &  $ 8.84$   &  $ 4.3\times 10^{+1}$ &   $ 19.91$  \\
			\hline
			4 &  $6.81 \times 10^{-1}$ & $ 5.30\times 10^{-12}$   &  $ 1.10\times 10^{+1}$ &  $ 2.58\times 10^{-1}$  &  $ 2.51\times 10^{+1}$ & $ 1.01\times 10^{-1}$  \\
			\hline
			6 &  $ 3.87\times 10^{-13}$ &  $2.15 \times 10^{-13}$  &  $ 1.43\times 10^{-3}$ &  $ 1.52\times 10^{-3}$  &  $ 3.48\times 10^{-3}$ & $ 8.03\times 10^{-3}$  \\
			\hline
			8 &  $1.10 \times 10^{-14}$ & $ 2.68\times 10^{-14}$  &  $ 3.10\times 10^{-5}$ &  $ 3.34\times 10^{-4}$  &  $ 8.38\times 10^{-5}$ & $ 1.83\times 10^{-3}$  \\
			\hline
			10 &  $ 8.75\times 10^{-15}$ &  $7.01 \times 10^{-15}$  &  $ 2.12\times 10^{-6}$ &  $1.12 \times 10^{-4}$  &  $ 1.0\times 10^{-5}$ & $ 6.25\times 10^{-4}$  \\
			\hline
			%
		\end{tabular}
	}
\end{table}

\newpage

Tables \ref{Table: PG Distributed-Order Results 1} and \ref{Table: PG Distributed-Order Results 2} show the convergence behaviour of the simulation results based on the PG scheme for the two case-I and II respectively. Tables \ref{Table: col. Distributed-Order Results 1} shows the convergence behaviour of the simulation results based on the collocation scheme for the two cases I and II. Table \ref{Table: col. Distributed-Order Results 2} shows the convergence behaviour of the simulation results based on the collocation scheme for the case where the exact solution is the same as case-I but the distribution function is $\phi(\alpha) = \sum_{p=1}^{4} \, \delta(\alpha-\alpha_p)$ with the fractional orders $\lbrace 1/10,1/2,13/10,19/10 \rbrace$ and the forcing function is $f(t) = \sum_{p=1}^{4} \,\, \frac{120}{\Gamma(6-\alpha_p)} \,\, t^{5-\alpha_p}$.

\begin{table}[h]
	\center
	\caption{\label{Table: col. Distributed-Order Results 2} Multi-term case; collocation scheme convergence study in $L^{\infty}$-norm, where $T=2$.}
	\vspace{-0.0 in}
	\scalebox{0.75}{
		\begin{tabular}
			{ l  c  c  c }
			\hline
			N & $\mu = 1 - 10^{-10}$ & {$\mu = 7/10$} & $\mu = 1/10$  \\
			&&& \\[-2.5ex]
			\hline			
			\hline
			6 &  $ 2.99596 \times 10^{-4}$ & $ 1.93088 \times 10^{3}$   &  $ 7.10859 \times 10^{-2}$ \\
			\hline
			10 &  $ 4.40056 \times 10^{-7}$ & $ 7.90916 \times 10^{-6}$   &  $ 1.95735 \times 10^{-4}$ \\
			\hline
			14 &  $ 9.35031\times 10^{-9}$ & $ 3.39228 \times 10^{-7}$   &  $ 7.99603 \times 10^{-6}$ \\
			\hline
			18 &  $ 9.15918\times 10^{-10}$ & $ 3.53369 \times 10^{-8}$   &  $ 8.27226 \times 10^{-7}$ \\
			\hline
			%
		\end{tabular}
	}
\end{table}
%


%
\subsection{Non-Smooth Solutions}
\label{Sec: Numerical Simulation - singular} 
%

Since the exact solution is not always known and in contrast to the standard fractional ODEs where the forcing term gives some regularity  information about the exact solution, in distributed order problems such a prediction is rather difficult to make. Hence, the fractional parameter $\mu$ can play the role of a \textit{fine-tuning knob} giving the possibility of searching for the best/optimal case, where the highest rate can be achieved with minimal degrees of freedom.
Here, we let $\alpha \in [0,1]$ and consider the following two cases of singular solution, where by the proper choice of $\mu$ we can easily capture the singularity of the solution.
\begin{itemize}
	\item 	Case III: $u^{ext} = t^{\mu_0}$, \, $\phi(\alpha) = \Gamma(1+\mu_0-\alpha)/(\mu_0)!$, \, $\mu_0 = 1/10, 9/10$,
	
	\item 	Case IV: $u^{ext} = t^{\mu_0} \, sin(t)$, \, $\phi(\alpha)$, \, $\mu_0 = 75/100 , 25/100$.
\end{itemize}
In case-III, we are able to obtain the exact solution only with one term by choosing $\mu = \mu_0$. In case-IV, we take $\mu = \mu _0$ and expand $sin(t)$ using Taylor series. Table \ref{Table: PG Distributed-Order Results nonsmooth} shows the $L^{2}$-norm convergence of the PG scheme for two different distribution functions.
\begin{table}[h]
	\center
	\caption{\label{Table: PG Distributed-Order Results nonsmooth} Case-IV; PG scheme convergence study in $L^{2}$-norm, where $T=2$.}
	\vspace{-0.1in}
	\scalebox{0.7}{
		\begin{tabular}
			{ l  c c c c}
			\hline
			& \multicolumn{2}{c}{$\mu = 75/100 $} & \multicolumn{2}{c}{$\mu = 25/100$}   \\
			&&&& \\[-2.5ex]
			\hline
			N & $\phi(\alpha) = 1$ & $\phi(\alpha) = $ Normal & $\phi(\alpha) = 1$ & $\phi(\alpha) = $ Normal \\
			\hline
			\hline
			2 &  $ 1.56682 \times 10^{-1}  $  & $ 1.62765 \times 10^{-1} $ & $ 1.5773 \times 10^{-1} $ & $ 1.548 \times 10^{-1} $ \\
			\hline
			4 &  $ 3.13043 \times 10^{-3} $  & $ 3.3898 \times 10^{-3} $ & $ 3.4228 \times 10^{-3} $ & $ 3.28626 \times 10^{-3} $ \\
			\hline
			6 &  $ 2.55359 \times 10^{-5} $  & $ 2.81522 \times 10^{-5} $ & $ 2.8956 \times 10^{-5} $ & $ 2.76729 \times 10^{-5} $ \\
			\hline
			8 &  $ 1.13562 \times 10^{-7} $  & $ 1.2512 \times 10^{-7} $ & $ 4.24126 \times 10^{-7} $ & $ 1.40114 \times 10^{-7} $ \\
			\hline
			10 &  $ 2.60471 \times 10^{-9} $  & $ 7.84647 \times 10^{-10} $ & $ 3.9524 \times 10^{-7} $ & $ 5.49882 \times 10^{-8} $ \\
			\hline
			%
		\end{tabular}
	}
\end{table}
%

%
\subsection{Condition Number}
\label{Sec: condition number} 
%
The condition number of the constructed linear system is obtained for different distribution functions, $\phi(\alpha)$. Tables \ref{Table: condition number PG} and \ref{Table: condition number Col.} show, respectively, the condition number of the constructed linear system for case-I and II based on PG and collocation scheme for the aforementioned distribution functions. 

\begin{table}[h]
	\center
	\caption{\label{Table: condition number PG} Case-I and II; PG scheme condition number of the constructed linear system, where $T= 2$.}
	\vspace{-0.1 in}
	\scalebox{0.7}{
		\begin{tabular}
			{ l  c c  c c  c c  c c }
			\hline
			& \multicolumn{2}{c}{$\mu = 2 - 10^{-8}$} & \multicolumn{2}{c}{$\mu = 2 - 10^{-1}$} & \multicolumn{2}{c}{$\mu = 1 + 1/2$} & \multicolumn{2}{c}{$\mu = 1 + 1/10$} \\
			\hline
			&&&&&&&& \\[-2.5ex]
			N & Case-I & Case-II  & Case-I & Case-II & Case-I & Case-II & Case-I & Case-II  \\
			\hline
			\hline
			6 &  $ 29706.682 $ & $ 4863.50 $  &  $ 14319.465 $  & $ 2168.87 $ &  $ 661.70145 $ & $ 70.1081 $ &  $ 51.928935 $ & $ 7.19267 $  \\
			\hline
			10 &  $ 240000.55 $ & $ 33494.8 $ &  $ 90197.388 $ & $ 11817.9 $ &  $ 1339.2855 $ & $ 130.925 $ &  $ 42.754111 $ & $ 9.35597 $  \\
			\hline
			14 &  $ 882010.62 $ & $ 118283 $ &  $ 279501.78 $ & $ 35395.6 $ &  $ 1941.9838 $ & $ 190.309 $ &  $ 47.335770 $ & $ 13.0337 $  \\
			\hline
			18 &  $ 2.2811229 \times 10^{6} $ & $ 301479 $ &  $ 633307.06 $ & $ 79324.0 $ &  $ 2505.5107 $ & $ 247.627 $ &  $ 50.713428 $ & $ 16.7944 $  \\
			\hline
			%
		\end{tabular}
	}
\end{table}
\begin{table}[h]
	\center
	\caption{\label{Table: condition number Col.} Case-I and II; collocation scheme condition number of the constructed linear system, where $T= 2$.}
	\vspace{-0.1 in}
	\scalebox{0.7}{
		\begin{tabular}
			{ l  c c  c c  c c  c c }
			\hline
			& \multicolumn{2}{c}{$\mu = 1- 10^{-8}$} & \multicolumn{2}{c}{$\mu = 1- 10^{-1}$} & \multicolumn{2}{c}{$\mu = 1/2$} & \multicolumn{2}{c}{$\mu = 1/10$} \\
			\hline
			&&&&&&&& \\[-2.5ex]
			N & Case-I & Case-II  & Case-I & Case-II & Case-I & Case-II & Case-I & Case-II  \\
			\hline
			\hline
			6 &  $ 67.5606 $ & $ 345.045 $  &  $ 60.3467 $  & $ 302.74 $ &  $ 43.6649 $ & $ 191.058 $ &  $ 36.0056 $ & $ 340.539 $  \\
			\hline
			10 &  $ 386.339 $ & $ 2781.51 $ &  $ 325.037 $ & $ 2309.25 $ &  $ 214.935 $ & $ 1515.11 $ &  $ 202.826 $ & $ 3554.95 $  \\
			\hline
			14 &  $ 1330.11 $ & $ 10646.1 $ &  $ 1076.14 $ & $ 8518.56 $ &  $ 685.166 $ & $ 5435.91 $ &  $ 713.002 $ & $ 16539.5 $  \\
			\hline
			18 &  $ 3388.95 $ & $ 28619.5 $ &  $ 2665.32 $ & $ 22290.4 $ &  $ 1661.16 $ & $ 13964.9 $ &  $ 3397.2 $ & $ 50911.1 $  \\
			\hline
			%
		\end{tabular}
	}
\end{table}

Moreover, three cases for the distribution function are considered: 1) the distribution is more biased toward the left of domain, 2) the distribution is symmetric  and 3) the distribution is more biased toward the right of domain, namely \textit{left biased}, \textit{symmetric} and \textit{right biased} respectively. The distribution functions are well-known  normal, exponential, log-normal, Cauchy, Laplace, Beta and Maxwell distributions, however, they are truncated and normalized, see Fig.\ref{Fig: Distribution functions}. For these distributions, the condition number of the constructed linear system based on the two methods is computed and provided in Tables \ref{Table: condition number left biased distribution}, \ref{Table: condition number sym distribution} and \ref{Table: condition number right biased distribution}.

\begin{figure}[h]
	\centering
	\begin{subfigure}{0.32\textwidth}
		\centering
		\includegraphics[width=1\linewidth]{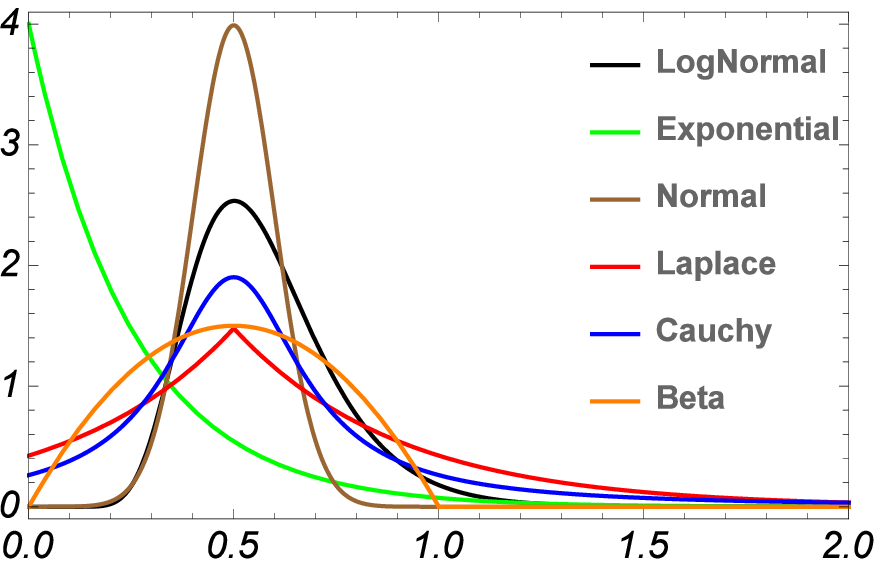}
		\caption{}
		\label{fig:sub1}
	\end{subfigure}
	\begin{subfigure}{0.32\textwidth}
		\centering
		\includegraphics[width=1\linewidth]{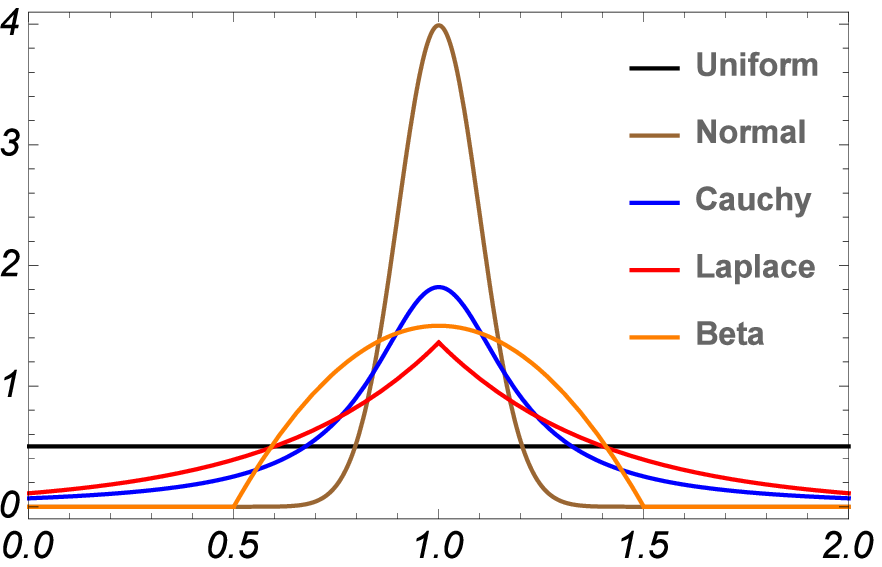}
		\caption{}
		\label{fig:sub2}
	\end{subfigure}
	\begin{subfigure}{0.32\textwidth}
		\centering
		\includegraphics[width=1\linewidth]{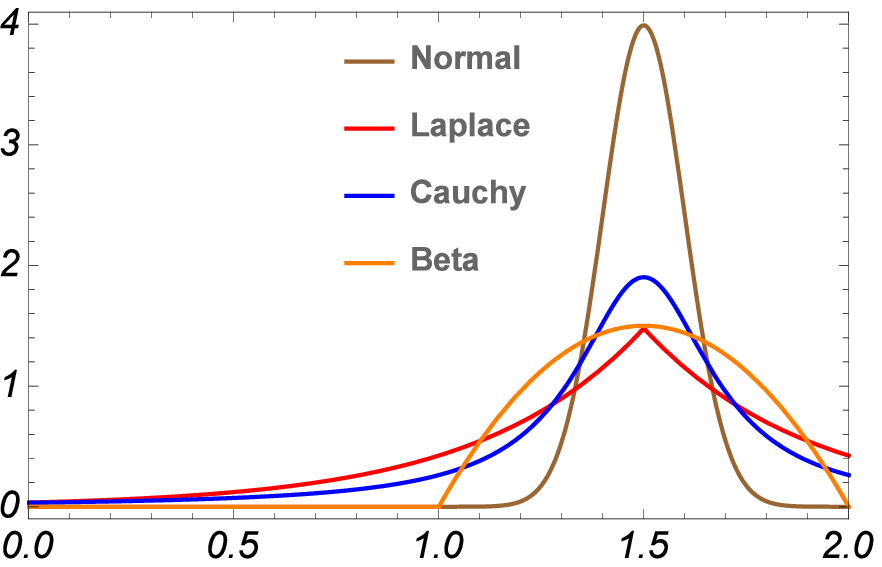}
		\caption{}
		\label{fig:sub3}
	\end{subfigure}
	\caption{\small{Distribution functions: (a) \textit{Left biased} (b) \textit{Symmetric} (c) \textit{Right biased} }}
	\label{Fig: Distribution functions}
\end{figure}
\begin{table}[h]
	\center
	\caption{\label{Table: condition number left biased distribution} Left biased distribution function; PG (top) and collocation (bottom) scheme condition number of the constructed linear system, where $T= 2$.}
	\vspace{-0.1 in}
	\scalebox{0.7}{
		\begin{tabular}
			{ l  c c  c c  c c }
			\hline
			&&&&&& \\[-2.5ex]
			\textbf{N} & \textbf{LogNormal} &  \textbf{Exponential} & \textbf{Normal} & \textbf{Laplace} & \textbf{Cauchy} & \textbf{Beta}  \\
			\hline
			\hline
			\multicolumn{7}{c}{$\mu = 2 - 1/10$} \\
			\hline
			6 &  $ 62101.5 $ & $ 130227 $  &  $ 85410.4 $  & $ 21714.7 $ &  $ 31361 $ & $ 70458.4 $   \\
			\hline
			10 &  $ 1.28119 \times 10^{6} $ & $ 2.62266 \times 10^{6} $ &  $ 2.15167 \times 10^{6} $ & $ 186527 $ &  $ 294630  $ & $ 1.51681 \times 10^{6} $   \\
			\hline
			14 &  $ 9.84911 \times 10^{6} $ & $ 1.61563 \times 10^{7} $ &  $ 1.98724 \times 10^{7} $ & $ 668178  $ &  $ 1.04066 \times 10^{6}  $ & $ 1.22216 \times 10^{6} $   \\
			\hline
			18 &  $ 4.48721 \times 10^{7} $ & $ 5.34428 \times 10^{7} $ &  $ 1.0748 \times 10^{8} $ & $ 1.62018 \times 10^{6}  $ &  $ 2.41399 \times 10^{6} $ & $ 5.86944 \times 10^{6} $   \\
			\hline
			\hline
			\multicolumn{7}{c}{$\mu = 1 + 1/10$} \\
			\hline
			6 &  $ 200.626 $ & $ 505.679 $  &  $ 300.643  $  & $ 71.5348 $ &  $ 100.504 $ & $ 233.849 $   \\
			\hline
			10 &  $ 654.259 $ & $ 1397.74 $ &  $ 1309.12 $ & $ 91.8644 $ &  $  140.467 $ & $ 816.254 $   \\
			\hline
			14 &  $ 1322.63 $ & $ 1969.15 $ &  $ 3437.72 $ & $ 64.4093  $ &  $ 98.4193  $ & $ 1780.96  $   \\
			\hline
			18 &  $ 2145.61 $ & $ 2178 $ &  $ 7037.99 $ & $ 70.7541  $ &  $  91.8767 $ & $ 3134.64  $   \\
			\hline			
			\vspace{0.1 in}
		\end{tabular}
	}
	\scalebox{0.7}{
		\begin{tabular}
			{ l@{\hspace*{.350in}} c@{\hspace*{.350in}} c@{\hspace*{.350in}}  c@{\hspace*{.350in}} c@{\hspace*{.350in}}  c@{\hspace*{.350in}} c@{\hspace*{.350in}} }
			\hline
			&&&&&& \\[-2.5ex]
			\textbf{N} & \textbf{LogNormal} &  \textbf{Exponential} & \textbf{Normal} & \textbf{Laplace} & \textbf{Cauchy} & \textbf{Beta}  \\
			\hline
			\hline
			\multicolumn{7}{c}{$\mu = 1 - 1/10$} \\
			\hline
			6 &  $ 20.1001 $ & $ 9.42989 $  &  $ 13.3574 $  & $ 51.9103 $ &  $ 39.5064 $ & $ 16.8765 $   \\
			\hline
			10 &  $ 48.3364 $ & $ 25.6229 $ &  $ 26.2852 $ & $ 237.89 $ &  $ 169.114 $ & $ 38.3283 $   \\
			\hline
			14 &  $ 91.0866 $ & $ 55.5712 $ &  $ 46.5097 $ & $ 714.563  $ &  $ 503.229 $ & $ 65.4666  $   \\
			\hline
			18 &  $ 143.171 $ & $ 126.19 $ &  $ 73.3388 $ & $ 1672.61  $ &  $ 1185.15  $ & $ 97.5852  $   \\
			\hline
			\hline
			\multicolumn{7}{c}{$\mu =  1/10$} \\
			\hline
			6 &  $ 15.7319 $ & $ 5.67323 $  &  $ 11.4277 $  & $ 33.4341 $ &  $ 26.3924 $ & $ 13.3573 $   \\
			\hline
			10 &  $ 40.2359 $ & $ 18.0958 $ &  $ 22.8619 $ & $ 150.375 $ &  $ 102.285  $ & $ 30.7949 $   \\
			\hline
			14 &  $ 71.0054 $ & $ 37.1664 $ &  $ 35.1324 $ & $ 435.867 $ &  $ 309.089  $ & $ 52.9303  $   \\
			\hline
			18 &  $ 110.725 $ & $ 80.9506 $ &  $ 48.0051 $ & $ 1032.88  $ &  $ 742.025  $ & $ 81.7202  $   \\
			\hline						
			%
		\end{tabular}
	}
\end{table}
\begin{table}[h]
	\center
	\caption{\label{Table: condition number sym distribution} Symmetric distribution function; PG (left) and collocation (right) scheme condition number of the constructed linear system, where $T= 2$.}
	\vspace{-0.1 in}
	\scalebox{0.6}{
		\begin{tabular}
			{ l  c c  c c  c  }
			\hline
			&&&&& \\[-2.5ex]
			\textbf{N} & \textbf{Uniform} &  \textbf{Normal} & \textbf{Cauchy} & \textbf{Laplace} & \textbf{Beta}  \\
			\hline
			\hline
			\multicolumn{6}{c}{$\mu = 2 - 1/10$} \\
			\hline
			6 &  $ 3104.6331 $ & $ 16261.6 $  &  $ 9328.53 $  & $ 6969.22 $ &  $ 13404.7 $   \\
			\hline
			10 &  $ 17244.219 $ & $ 220205 $ &  $ 80110.2 $ & $ 51304.3 $ &  $ 157767  $   \\
			\hline
			14 &  $ 52095.460 $ & $ 1.28827 \times 10^{6} $ &  $ 308549 $ & $ 179737  $ &  $ 817268  $   \\
			\hline
			18 &  $ 117338.89  $ & $ 4.86517 \times 10^{6} $ &  $ 803362 $ & $ 441628 $ &  $ 2.77821 \times 10^{6}  $   \\
			\hline
			\hline
			\multicolumn{6}{c}{$\mu = 1 + 1/10$} \\
			\hline
			6 &  $ 9.5451677 $ & $ 44.0857 $  &  $ 27.6858 $  & $ 17.9872 $ &  $ 37.4224 $   \\
			\hline
			10 &  $ 11.211269 $ & $ 85.3822 $ &  $ 30.0092 $ & $ 21.361 $ &  $ 64.2331  $   \\
			\hline
			14 &  $ 12.360897 $ & $ 127.844 $ &  $ 30.9343 $ & $ 20.1997  $ &  $ 87.2794  $   \\
			\hline
			18 &  $ 15.900925 $ & $ 172.888 $ &  $ 34.2126 $ & $ 22.1982  $ &  $  107.403 $   \\
			\hline						
			%
		\end{tabular}
	}
	\scalebox{0.6}{
		\begin{tabular}
			%
			{ l  c c  c c  c }
			\hline
			&&&&& \\[-2.5ex]
			\textbf{N} & \textbf{Uniform} &  \textbf{Normal} & \textbf{Cauchy} & \textbf{Laplace} & \textbf{Beta}  \\
			\hline
			\hline
			\multicolumn{6}{c}{$\mu = 1 - 1/10$} \\
			\hline
			6 &  $ 219.733 $ & $ 81.9543 $  &  $ 118.433 $  & $ 140.922 $ &  $ 95.0983 $  \\
			\hline
			10 &  $ 1592.05 $ & $ 284.171 $ &  $ 598.247 $ & $ 798.174 $ &  $ 375.673  $  \\
			\hline
			14 &  $ 5769.24 $ & $ 620.056 $ &  $ 1788.85 $ & $ 2549.17  $ &  $ 905.972  $  \\
			\hline
			18 &  $ 14944.9 $ & $ 1097.37 $ &  $ 4122.7 $ & $ 6115.65  $ &  $ 1737.18  $  \\
			\hline
			\hline
			\multicolumn{6}{c}{$\mu = 1/10$} \\
			\hline
			6 &  $ 183.984 $ & $ 55.8908 $  &  $ 72.0876 $  & $ 78.4252 $ &  $ 62.069 $  \\
			\hline
			10 &  $ 1854.46 $ & $ 198.226 $ &  $ 360.943 $ & $ 506.994 $ &  $ 248.186  $  \\
			\hline
			14 &  $ 8678.37 $ & $ 439.316 $ &  $ 1113.83 $ & $ 1595.84 $ &  $ 589.611  $  \\
			\hline
			18 &  $ 26873.1 $ & $ 786.655 $ &  $ 2600.56 $ & $ 3954.72  $ &  $ 1171.64 $  \\
			\hline						
			%
		\end{tabular}
	}
\end{table}
\begin{table}[h]
	\center
	\caption{\label{Table: condition number right biased distribution} Right biased distribution function; PG (left) and collocation (right) scheme condition number of the constructed linear system, where $T= 2$.}
	\vspace{-0.1 in}
	\scalebox{0.7}{
		\begin{tabular}
			%
			{ l  c c  c c }
			\hline
			&&&& \\[-2.5ex]
			\textbf{N} & \textbf{Normal} & \textbf{Laplace} & \textbf{Cauchy} & \textbf{Beta}  \\
			\hline
			\hline
			\multicolumn{5}{c}{$\mu = 2 - 1/10$} \\
			\hline
			6 &  $ 2985.08 $ & $ 2274.39 $  &  $ 2368.82 $  & $ 2105.03 $   \\
			\hline
			10 &  $ 23439.3 $ & $  13948.9 $ &  $ 15471.9 $ & $ 13474.7 $   \\
			\hline
			14 &  $ 92925.7 $ & $ 45201.4 $ &  $ 52626.9 $ & $ 45325.7  $    \\
			\hline
			18 &  $ 259993 $ & $ 107082 $ &  $ 129454 $ & $ 110813 $   \\
			\hline
			\hline
			\multicolumn{5}{c}{$\mu = 1 + 1/10$} \\			
			\hline
			6 &  $ 6.99801 $ & $ 6.60933 $  &  $ 6.54308 $  & $ 6.18762 $   \\
			\hline
			10 &  $ 10.8049 $ & $ 9.81523 $ &  $ 10.0861 $ & $ 10.1067 $     \\
			\hline
			14 &  $ 15.1563 $ & $ 14.2958 $ &  $ 14.0265 $ & $ 14.6423  $   \\
			\hline
			18 &  $ 19.4415 $ & $ 18.4003 $ &  $ 18.8445 $ & $ 18.8791  $  \\
			\hline
			%
		\end{tabular}
	}
	\scalebox{0.7}{
		\begin{tabular}
			%
			{ l  c c  c c}
			\hline
			&&&& \\[-2.5ex]
			\textbf{N} & \textbf{Normal} & \textbf{Laplace} & \textbf{Cauchy} & \textbf{Beta}  \\
			\hline
			\hline
			\multicolumn{5}{c}{$\mu = 1 - 1/10$} \\
			\hline
			6 &  $ 329.01 $ & $ 328.433 $  &  $ 333.499 $  & $ 360.053 $    \\
			\hline
			10 &  $ 2022.09 $ & $ 2330.68 $ &  $ 2275.43 $ & $ 2501.41 $    \\
			\hline
			14 &  $ 6299.73 $ & $ 8170.64 $ &  $ 7736.87 $ & $ 8579.48  $    \\
			\hline
			18 &  $ 14429.7 $ & $ 20615.9 $ &  $ 19067.6 $ & $ 21250  $   \\
			\hline
			\hline
			\multicolumn{5}{c}{$\mu =  1/10$} \\			
			\hline
			6 &  $ 278.525 $ & $ 333.384 $  &  $ 322.187 $  & $ 378.704 $    \\
			\hline
			10 &  $ 1539.25 $ & $ 2647.09 $ &  $ 2246.71 $ & $ 2727.03 $    \\
			\hline
			14 &  $ 4701.38 $ & $ 10365.3 $ &  $ 8041.12 $ & $ 9884.8 $    \\
			\hline
			18 &  $ 10622.6 $ & $ 28764.7 $ &  $ 20895.4 $ & $ 25800.8  $   \\
			\hline			
			%
		\end{tabular}
	}
\end{table}
%


\newpage

%
\section{Summary and Discussion}
\label{Sec: Summary and Conclusion} 
%

We developed two spectrally-accurate schemes, namely the Petrov-Galerkin spectral method and the fractional spectral collocation method for distributed order fractional differential equations. The two schemes were constructed based on the recently developed spectral theory for fractional Sturm-Liouville problems (FSLPs). In the Petrov-Galerkin method, we employed the Jacobi poly-fractonomials as the bases, which are the eigenfunctions of FSLP-I, and the poly-fractonomial eigenfunctions of FSLP-II as the test functions. We carried out the discrete stability analysis of the proposed scheme employing some equivalent/bilinear-induced norms based on the defined distributed Sobolev spaces and their associated norms. In addition, we performed a convergence study of the proposed scheme. In the collocation method, we employed fractional Lagrange interpolants satisfying the Kronecker delta property at the collocation points, and then we obtained the corresponding distributed differentiation matrices to discretize the strong problem.

The existing schemes in the literature are mostly employing finite difference methods. The main challenge in these methods, in comparison to spectral methods, is the history calculation as well as extensive memory allocation while they deliver fixed algebraic accuracies. The recent spectral theory on fractional Sturm-Liouville problems (FSLPs) in \cite{Zayernouri2013} naturally motivates the use of Petrov-Galerkin spectral methods, where the arising bilinear forms are comprised of left- and right-sided fractional derivatives. The eigen-functions of FSLPs can be employed naturally as the bases and test spaces, where their left- and right-sided derivatives are obtained analytically. These functions consist of a polynomial part and a fractional part, where the former leaves the fractional order, $\mu$, as a free parameter to capture solution singularities, hence, to tune up the accuracy of the scheme from being algebraically convergent to exponential convergent. In fact, the Case-III of numerical examples demonstrated how a proper choice of fractional part of the bases provides the exact solution with only one term expansion. Furthermore, we proved that the distributed bilinear form can be approximated with a spectral/exponential accuracy using a proper quadrature rule. The PG spectral method treats the nonlocal effects efficiently through a global spectral method and provides a nice mathematical framework for performing theoretical studies, however, treating nonlinear problems remains a challenge. To this end, we constructed a spectrally accurate fractional spectral collocation method employing fractional Lagrange interpolants, where for linear problems the two developed schemes become equivalent in terms of the rate of convergence. 

The distribution function, $\phi(\alpha)$, defined the distribution of the differentiation fractional-order, $\alpha$, and it could arbitrarily confine the domain over which the fractional differentiation is taken. If $\phi$ was integrable in a compact support in $[\alpha_{min},\alpha_{max}]$, then $H^{\alpha_{min}}(\mathbb{R}) \supseteq \prescript{\phi}{}{\mathcal H}(\mathbb{R}) \supseteq H^{\alpha_{max}}(\mathbb{R})$. Hence, $\phi$ could play a crucial rule in defining the underlying solution space properly. In anomalous physical processes, the distribution function can be obtained from experimental data, where the inherent data uncertainty can be incorporated through the $\phi$ obtained from the observed data, hence, leading to a robust data-driven simulation framework for multi-physics problems.

\appendix
%
\section{Proof Of Lemma \eqref{Lem: Frac derive of Poly-fractonomials}}
\label{Sec: App. proof frac derive of Poly-fractonomials} 
%
\begin{proof}
	Following \cite{Askey1969} and for $\sigma>0$, $\alpha >-1$, $\beta >-1$, and $\forall x \in [-1,1]$ we have
	\begin{equation}
		\label{Eq: Askey1_Original}
		(1+x)^{\beta + \sigma}\, \frac{P_n^{\alpha-\sigma,\beta+\sigma}(x)}{P_n^{\alpha-\sigma,\beta+\sigma}(-1)} =
		\frac{\Gamma(\beta + \sigma +1)   }{ \Gamma(\beta +1)\Gamma(\sigma) P_n^{\alpha,\beta}(-1)}
		\int_{-1}^{x} 
		\frac{(1+s)^{\beta}\,\,P_n^{\alpha,\beta}(s)   }{(x - s)^{1-\sigma}}ds,
	\end{equation}
	and 
	\begin{equation}
		\label{Eq: Askey2_Original}
		(1-x)^{\alpha + \sigma}\, \frac{P_n^{\alpha+\sigma,\beta -\sigma}(x)}{P_n^{\alpha +\sigma,\beta-\sigma}(+1)} =
		\frac{\Gamma(\alpha + \sigma +1)   }{ \Gamma(\alpha +1)\Gamma(\sigma) P_n^{\alpha,\beta}(+1)}
		\int_{x}^{1} 
		\frac{(1-s)^{\alpha}\,\,P_n^{\alpha,\beta}(s)   }{(s - x)^{1-\sigma}}ds.
	\end{equation}
	%
	By the definition of the left-sided Riemann-Liouville integral $\prescript{RL}{-1}{\mathcal{I}}_{x}^{\sigma} $ and evaluating the special end-values $P_n^{\alpha-\sigma,\beta+\sigma}(-1)$ and $P_n^{\alpha,\beta}(-1)$, we can re-write (\ref{Eq: Askey1_Original}) as 
	\begin{equation}
		\nonumber
		\prescript{RL}{-1}{\mathcal{I}}_{x}^{\sigma} \Big\lbrace (1+x)^{\beta} P_n^{\alpha,\beta}(x)  \Big\rbrace = \frac{\Gamma(n +\beta  + 1)}{\Gamma(n + \beta + \sigma + 1)} (1+x)^{\beta + \sigma}\,P_n^{\alpha-\sigma,\beta+\sigma}(x),
	\end{equation}
	where, by taking the fractional derivative $\prescript{RL}{-1}{\mathcal{D}}_{x}^{\sigma}$ on the both sides, 
	we obtain
	\begin{equation}
		\label{Eq: Askey1_Fractional_Derivative}
		\prescript{RL}{-1} {\mathcal{D}}_{x}^{\sigma} \Big\lbrace
		(1+x)^{\beta + \sigma}
		P_n^{\alpha-\sigma, \beta + \sigma} (x) \Big\rbrace  = \frac{\Gamma(n +\beta + \sigma + 1)}{\Gamma(n +\beta +1)}  (1+x)^{\beta} P^{ \alpha , \beta}_n(x).  
	\end{equation}
	Hence, taking $\beta + \sigma = \mu$, $\alpha - \sigma = -\mu$ in \eqref{Eq: Askey1_Fractional_Derivative}, and shifting from $n$ to $n-1$, we obtain
	\begin{eqnarray}
		\label{Eq: 1 Presenving structure}
		\prescript{RL}{-1} {\mathcal{D}}_{x}^{\sigma} 
		\Big\lbrace
		\prescript{(1)}{}{ \mathcal{P}}_{n}^{\mu}(x)
		\Big\rbrace  &=& \frac{\Gamma(n +\mu )}{\Gamma(n +\mu - \sigma )}  (1+x)^{\mu - \sigma} P^{ \sigma - \mu , \mu - \sigma}_{n-1}(x),
		\\ \nonumber
		&=&
		\frac{\Gamma(n +\mu )}{\Gamma(n +\mu - \sigma )}  (1+x)^{\eta} P^{ -\eta , \eta}_{n-1}(x),
		\\ \nonumber
		&=&
		\frac{\Gamma(n +\mu )}{\Gamma(n +\mu - \sigma)}  \prescript{(1)}{}{ \mathcal{P}}_{n}^{\eta}(x),
	\end{eqnarray}
	where $\eta = \mu - \sigma$. Moreover, by the definition of the right-sided Riemann-Liouville integral $\prescript{RL}{x}{\mathcal{I}}_{1}^{\sigma}$ and evaluating the special end-values $P_n^{\alpha-\sigma,\beta+\sigma}(+1)$ and $P_n^{\alpha,\beta}(+1)$, we can re-write (\ref{Eq: Askey2_Original}) as 
	\begin{equation}
		\nonumber
		\prescript{RL}{x}{\mathcal{I}}_{1}^{\sigma} \Big\lbrace (1-x)^{\alpha} P_n^{\alpha,\beta}(x)  \Big\rbrace = \frac{\Gamma(n + \alpha + 1)}{\Gamma(n + \alpha + \sigma + 1)} (1-x)^{\alpha+\sigma}  P_n^{\alpha+\sigma,\beta-\sigma}(x). 
	\end{equation}
	In a similar fashion, by taking the fractional derivative $\prescript{RL}{x}{\mathcal{D}}_{-1}^{\sigma}$ on the both sides,
	we obtain
	\begin{equation}
		\label{Eq: Askey2_Fractional_Derivative}
		\prescript{RL}{x}{\mathcal{D}}_{1}^{\sigma} \Big\lbrace
		(1-x)^{\alpha + \sigma}
		P_n^{\alpha + \sigma , \beta-\sigma} (x) \Big\rbrace = \frac{\Gamma(n +\alpha + \sigma + 1)}{\Gamma(n +\alpha +1)}  (1 - x)^{\alpha} P^{\alpha , \beta}_n(x).  
	\end{equation}
	Next, by taking $\alpha + \sigma = \mu$, $\beta - \sigma = -\mu$ in \eqref{Eq: Askey2_Fractional_Derivative}, and again shifting from $n$ to $n-1$ we have
	\begin{eqnarray}
		\label{Eq: 2 Presenving structure}
		\prescript{RL}{x}{\mathcal{D}}_{1}^{\sigma} \Big\lbrace
		\prescript{(2)}{}{ \mathcal{P}}_{n}^{\mu}(x)
		\Big\rbrace &=& \frac{\Gamma(n +\mu )}{\Gamma(n +\mu - \sigma)}  (1 - x)^{\mu-\sigma} P^{\mu - \sigma , \sigma - \mu}_{n-1}(x).  
		\\ \nonumber
		&=&
		\frac{\Gamma(n +\mu )}{\Gamma(n +\mu - \sigma )}  (1-x)^{\eta} P^{ \eta , -\eta}_{n-1}(x),
		\\ \nonumber
		&=&
		\frac{\Gamma(n +\mu )}{\Gamma(n +\mu - \sigma)}  \prescript{(2)}{}{ \mathcal{P}}_{n}^{\eta}(x),
	\end{eqnarray}
	and that completes the proof.
\end{proof}
%

%
\section{Proof Of Theorem \eqref{Thm: dis. norm equivalency}}
\label{Sec: App. proof dis. norm equivalency} 
%
\begin{proof}
	Let $\phi$ be bounded in $(\alpha_{min} , \alpha_{max} )$. Then,
	\begin{eqnarray}
		\nonumber
		C_1 \phi_{min} A \,\, \leq \,\, \Vert \cdot \Vert^2_{l, \phi, I} \,\, \leq \,\, C_2 \phi_{max} A,
		\\
		\label{Eq: dis. left and right norm equivalency}
		\\
		\nonumber
		C_3 \phi_{min} B \,\, \leq \,\, \Vert \cdot \Vert^2_{r, \phi, I} \,\, \leq C_4 \phi_{max} B,
	\end{eqnarray}
	where
	\begin{eqnarray}
		\nonumber
		A = 
		\Vert \cdot \Vert^2_{L^2(I)} + 
		\int_{\alpha_{min}}^{\alpha_{max}}  \norm{ \, \prescript{RL}{x_L}{\mathcal{D}}_{x}^{\alpha}(\cdot)  \, }^2_{L^2(I)} \,\, d\alpha,
		\\
		\nonumber
		B = 
		\Vert \cdot \Vert^2_{L^2(I)} + 
		\int_{\alpha_{min}}^{\alpha_{max}}  \norm{ \, \prescript{RL}{x}{\mathcal{D}}_{x_R}^{\alpha}(\cdot)  \, }^2_{L^2(I)} \,\, d\alpha ,
	\end{eqnarray}
	and $C_1$, $C_2$, $C_3$, and $C_4$ are positive constants.
	From \cite{Li2009}, we know that $\forall \alpha = s$ fixed, $\Vert \cdot \Vert_{l, s, I}  \sim   \Vert \cdot \Vert_{r, s, I} $ that is $ \norm{ \, \prescript{RL}{x_L}{\mathcal{D}}_{x}^{s}(\cdot)  \, }^2_{L^2(I)}  \sim   \norm{ \, \prescript{RL}{x}{\mathcal{D}}_{x_R}^{s}(\cdot)  \, }^2_{L^2(I)} $, hence let $\vec{s} = \lbrace s_1, s_2, \cdots , s_Q \rbrace$, and similarly $\forall s = s_q$ the aforementioned equivalence holds. Therefore, any linear combination of
	\begin{equation*}
		\sum_{q=1}^{Q} \, w_q \, \norm{ \, \prescript{RL}{x_L}{\mathcal{D}}_{x}^{s}(\cdot)  \, }^2_{L^2(I)}  
		\sim   
		\sum_{q=1}^{Q} \, w_q \, \norm{ \, \prescript{RL}{x}{\mathcal{D}}_{x_R}^{s}(\cdot)  \, }^2_{L^2(I)} .
	\end{equation*}
	Taking $Q \longrightarrow \infty$ and assuming $\lbrace w_q , s_q \rbrace$ to be Riemann integral weights and points in $[\alpha_{min} , \alpha_{max}]$, 
	\begin{equation*}
		\int_{\alpha_{min}}^{\alpha_{max}} \, \norm{ \, \prescript{RL}{x_L}{\mathcal{D}}_{x}^{s}(\cdot)  \, }^2_{L^2(I)}  
		\sim   
		\int_{\alpha_{min}}^{\alpha_{max}} \, \norm{ \, \prescript{RL}{x}{\mathcal{D}}_{x_R}^{s}(\cdot)  \, }^2_{L^2(I)} .
	\end{equation*}
	By adding $\Vert \cdot \Vert^2_{L^2(I)}$ to the both sides of the above equivalence, we obtain $A \sim B$; and by \eqref{Eq: dis. left and right norm equivalency}, $ \Vert \cdot \Vert_{l, \phi, I} \sim \Vert \cdot \Vert_{r, \phi, I}$.

	In addition, from \cite{Li2009}, we know that $\forall \alpha = s$ fixed, $\Vert \cdot \Vert^2_{s, I}  \sim  \Vert \cdot \Vert^2_{l, s, I}$. Let $\vec{s} = \lbrace s_1, s_2, \cdots , s_Q \rbrace$ thus $\forall s = s_q \in \vec{s},  \,\,  \Vert \cdot \Vert^2_{s_q, I} \sim  \Vert \cdot \Vert^2_{l, s_q, I}$. Therefore, for any linear combination of $\sum_{q=1}^{Q} \, w_q \, \Vert \cdot \Vert^2_{s_q, I} \sim \sum_{q=1}^{Q} \, w_q \, \Vert \cdot \Vert^2_{l, s_q, I}$. Taking $Q \longrightarrow \infty$, we obtain:
	\begin{equation*}
		\int_{\alpha_{min}}^{\alpha_{max}} \, \Vert \cdot \Vert^2_{\alpha, I} \, d\alpha 
		\,\, \sim \,\,
		\int_{\alpha_{min}}^{\alpha_{max}} \, \Vert \cdot \Vert^2_{l, \alpha, I} \, d\alpha ,
	\end{equation*}
	where the right hand side of the equivalence is
	\begin{eqnarray}
		\nonumber
		\int_{\alpha_{min}}^{\alpha_{max}} \left(\Vert \cdot \Vert^2_{L^2(I)} + \norm{ \, \prescript{RL}{x_L}{\mathcal{D}}_{x}^{\alpha}(\cdot)  \, }^2_{L^2(I)}\right) \,\, d\alpha
		&=&
		(\alpha_{max} - \alpha_{min}) \Vert \cdot \Vert^2_{L^2(I)} + \int_{\alpha_{min}}^{\alpha_{max}} \norm{ \, \prescript{RL}{x_L}{\mathcal{D}}_{x}^{\alpha}(\cdot)  \, }^2_{L^2(I)} \,\, d\alpha,
		\\
		\nonumber
		&\sim& \Vert \cdot \Vert^2_{l, \phi, I} \, \sim \, \Vert \cdot \Vert^2_{r, \phi, I}.
	\end{eqnarray}
	Therefore,
	\begin{equation}
		\label{Eq: dis. norm equivalency proof - temp1}
		\int_{\alpha_{min}}^{\alpha_{max}} \, \Vert \cdot \Vert^2_{\alpha, I} \, d\alpha 
		\,\, \sim \,\,
		\Vert \cdot \Vert^2_{l, \phi, I} \, \sim \, \Vert \cdot \Vert^2_{r, \phi, I}.
	\end{equation}
	We can also show that
	\begin{eqnarray}
		\nonumber
		\phi_{min}
		\int_{\alpha_{min}}^{\alpha_{max}} \, \norm{ \, (1+\vert \omega \vert^2 )^{\frac{\alpha}{2}}\mathcal{F}(\cdot)(\omega) \,  }^2_{L^2(\mathbb{R})} \,\, d\alpha
		\leq
		\Vert \cdot \Vert^2_{\phi, \mathbb{R}} 
		\leq
		\phi_{max}
		\int_{\alpha_{min}}^{\alpha_{max}} \, \norm{ \, (1+\vert \omega \vert^2 )^{\frac{\alpha}{2}}\mathcal{F}(\cdot)(\omega) \,  }^2_{L^2(\mathbb{R})} \,\, d\alpha .
		&
	\end{eqnarray}
	Because of the non-negativity of the norms, we have
	\begin{eqnarray}
		\nonumber
		&
		\phi_{min}
		\inf_{\tilde{v} \in \prescript{\phi}{}{\mathcal H}(\mathbb{R}), \tilde{v}|_{I}=(\cdot) }
		\int_{\alpha_{min}}^{\alpha_{max}} \, \norm{ \, (1+\vert \omega \vert^2 )^{\frac{\alpha}{2}}\mathcal{F}(\tilde{v})(\omega) \,  }^2_{L^2(\mathbb{R})} \,\, d\alpha
		&
		\\
		\label{Eq: dis. norm equivalency in interval I}
		&
		\leq
		\inf_{\tilde{v} \in \prescript{\phi}{}{\mathcal H}(\mathbb{R}), \tilde{v}|_{I}=(\tilde{v}) }
		\Vert \cdot \Vert^2_{\phi, \mathbb{R}} 
		\leq
		&
		\\
		\nonumber
		&
		\phi_{max}
		\inf_{\tilde{v} \in \prescript{\phi}{}{\mathcal H}(\mathbb{R}), \tilde{v}|_{I}=(\cdot) }
		\int_{\alpha_{min}}^{\alpha_{max}} \, \norm{ \, (1+\vert \omega \vert^2 )^{\frac{\alpha}{2}}\mathcal{F}(\tilde{v})(\omega) \,  }^2_{L^2(\mathbb{R})} \,\, d\alpha .
		&
	\end{eqnarray}
	In general, $\prescript{\phi}{}{\mathcal H}(\mathbb{R}) \subset H^{\alpha_{max}}(\mathbb{R})$, $\forall \alpha \in [\alpha_{min} , \alpha_{max}]$. Therefore, we have:
	\begin{eqnarray}
		\nonumber
		&&
		\inf_{\tilde{v} \in \prescript{\phi}{}{\mathcal H}(\mathbb{R}), \tilde{v}|_{I}=(\cdot) }
		\int_{\alpha_{min}}^{\alpha_{max}} \, \norm{ \, (1+\vert \omega \vert^2 )^{\frac{\alpha}{2}}\mathcal{F}(\tilde{v})(\omega) \,  }^2_{L^2(\mathbb{R})} \,\, d\alpha,
		\\
		\nonumber
		& = & 
		\int_{\alpha_{min}}^{\alpha_{max}} \, 
		\inf_{\tilde{v} \in \prescript{\phi}{}{\mathcal H}(\mathbb{R}), \tilde{v}|_{I}=(\cdot) } \,\,
		\norm{ \, (1+\vert \omega \vert^2 )^{\frac{\alpha}{2}}\mathcal{F}(\tilde{v})(\omega) \,  }^2_{L^2(\mathbb{R})} \,\, d\alpha,
		\\
		\nonumber
		& \leq &
		\mathcal{C}
		\int_{\alpha_{min}}^{\alpha_{max}} \, 
		\inf_{\tilde{v} \in H^{\alpha}(\mathbb{R}), \tilde{v}|_{I}=(\cdot) } \,\,
		\norm{ \, (1+\vert \omega \vert^2 )^{\frac{\alpha}{2}}\mathcal{F}(\tilde{v})(\omega) \,  }^2_{L^2(\mathbb{R})} \,\, d\alpha,
		\\
		\nonumber
		& = &
		\mathcal{C}
		\int_{\alpha_{min}}^{\alpha_{max}} \, 
		\norm{ \, \cdot \,  }^2_{\alpha , I} \,\, d\alpha.
	\end{eqnarray}
	However, for some choices of $\phi = \delta(\alpha - \alpha_{min})$ and thus $\prescript{\phi}{}{\mathcal H}(\mathbb{R}) \supset H^{\alpha_{max}}(\mathbb{R})$, $\forall \alpha \in [\alpha_{min} , \alpha_{max}]$. Therefore,
	\begin{eqnarray}
		\nonumber
		&&
		\int_{\alpha_{min}}^{\alpha_{max}} \, 
		\inf_{\tilde{v} \in \prescript{\phi}{}{\mathcal H}(\mathbb{R}), \tilde{v}|_{I}=(\cdot) } \,\,
		\norm{ \, (1+\vert \omega \vert^2 )^{\frac{\alpha}{2}}\mathcal{F}(\tilde{v})(\omega) \,  }^2_{L^2(\mathbb{R})} \,\, d\alpha,
		\\
		\nonumber
		& \geq &
		\widetilde{\mathcal{C}}
		\int_{\alpha_{min}}^{\alpha_{max}} \, 
		\inf_{\tilde{v} \in H^{\alpha}(\mathbb{R}), \tilde{v}|_{I}=(\cdot) } \,\,
		\norm{ \, (1+\vert \omega \vert^2 )^{\frac{\alpha}{2}}\mathcal{F}(\tilde{v})(\omega) \,  }^2_{L^2(\mathbb{R})} \,\, d\alpha,
		\\
		\nonumber
		& = &
		\widetilde{\mathcal{C}}
		\int_{\alpha_{min}}^{\alpha_{max}} \, 
		\norm{ \, \cdot \,  }^2_{\alpha , I} \,\, d\alpha,
	\end{eqnarray}
	which by \eqref{Eq: dis. norm equivalency in interval I} and \eqref{Eq: NormFrac Sobolev on Interval}, we get
	\begin{equation}
		\label{Eq: dis. norm equivalency proof - temp2}
		\int_{\alpha_{min}}^{\alpha_{max}} \, 
		\norm{ \, \cdot \,  }^2_{\alpha , I} \,\, d\alpha
		\sim
		\norm{ \, \cdot \,  }^2_{\phi , I}.
	\end{equation}
	Comparing \eqref{Eq: dis. norm equivalency proof - temp1} and \eqref{Eq: dis. norm equivalency proof - temp2}, we have 
	\begin{equation*}
		\norm{ \, \cdot \,  }^2_{\phi , I} \sim \norm{ \, \cdot \,  }^2_{l, \phi , I}  \sim \norm{ \, \cdot \,  }^2_{r, \phi , I}.
	\end{equation*}
\end{proof}
\begin{rem}
	%
	We note that if $\phi = \delta(\alpha - s)$, we recover the standard $\prescript{RL}{x_L}{\mathcal{D}}_{x}^{s}(u) = f$, where the equivalence between the corresponding $\Vert \cdot \Vert_{l, s, I}$, $\Vert \cdot \Vert_{r, s, I}$, and $\Vert \cdot \Vert_{s, I}$ has been already established. Moreover, we note that for the case $\phi \in L^1(\, [\alpha_{min} , \alpha_{max} \,)$ containing finitely many singularities at $\alpha_1, \alpha_2, \cdots, \alpha_m$, the whole interval $[\alpha_{min} , \alpha_{max}]$ and the integration can be written as
	\begin{equation*}
		\int_{\alpha_{min}}^{\alpha_1} \phi(\alpha) \, \prescript{RL}{x_L}{\mathcal{D}}_{x}^{s}(u) \, d\alpha
		+ \int_{\alpha_1}^{\alpha_2} \phi(\alpha) \, \prescript{RL}{x_L}{\mathcal{D}}_{x}^{s}(u) \, d\alpha
		+ \cdots 
		+ \int_{\alpha_m}^{\alpha_{max}} \phi(\alpha) \, \prescript{RL}{x_L}{\mathcal{D}}_{x}^{s}(u) \, d\alpha,
	\end{equation*}
	where all the previous steps in the proof can apply in each interval.
\end{rem}

%
\section{Proof of Lemma \eqref{Lem: fractional integ-by-parts 1 and 2}}
\label{Sec: App. proof fractional integ-by-parts 1 and 2} 
%

%
\begin{proof}
	Since $u(a) =0$, by \eqref{Eq: Property of RL deriv} $\prescript{}{a}{\mathcal{D}}_{x}^{2\mu} u = \prescript{}{a}{\mathcal{D}}_{x}^{\mu} (\prescript{}{a}{\mathcal{D}}_{x}^{\mu} u)$. Taking $\widetilde{u}(x) = \prescript{}{a}{\mathcal{D}}_{x}^{\mu} u$, we have
	\begin{eqnarray}
		\nonumber
		( \prescript{}{a}{\mathcal{D}}_{x}^{2\mu} u \,,\,v ) &=& ( \prescript{}{a}{\mathcal{D}}_{x}^{\mu} \widetilde{u} \,,\,v ),
		\\ \nonumber
		&=&
		\frac{1}{\Gamma(1-\mu)} \int_{a}^{b} \Big[  \frac{d}{dx} \int_{a}^{x}  \frac{\widetilde{u}(s) ds}{(x-s)^{\mu}}  \Big] v(x) dx,
		\\ \nonumber
		&=&
		\Big\{ \frac{v(x)}{\Gamma(1-\mu)} \int_{a}^{x}  \frac{\widetilde{u}(s) ds}{(x-s)^{\mu}}    \Big\}_{x=a}^{x=b}
		\\ \nonumber
		&-&
		\frac{1}{\Gamma(1-\mu)} \int_{a}^{b} \Big[ \int_{a}^{x}  \frac{\widetilde{u}(s) ds}{(x-s)^{\mu}}  \Big] \frac{dv(x)}{dx} dx, \quad by\,\, v(b)=0,
		\\ 
		\label{Eq: derivation integ-by-parts1}
		&=&
		-v(a)\,\lim_{x\rightarrow a} \prescript{}{a}{\mathcal{I}}_{x}^{1-\mu} \widetilde{u}
		-
		\frac{1}{\Gamma(1-\mu)} \int_{a}^{b}  \int_{s}^{b}  \frac{\frac{dv(x)}{dx} dx}{(x-s)^{\mu}}   \widetilde{u}(s) ds, 
	\end{eqnarray} 
	which make sense when the interior term $\int_{s}^{b}  \frac{\frac{dv(x)}{dx} dx}{(x-s)^{\mu}}$ is integrable in $(a,b)$. Taking into account that $v(a)$ is bounded, we can show that the boundary term $v(a)\,\lim_{x\rightarrow a} \prescript{}{a}{\mathcal{I}}_{x}^{1-\mu} \widetilde{u}$ also vanishes as 
	\begin{eqnarray}
		\label{Eq: fracinteg to zero}
		\lim_{x\rightarrow a} \Big| \prescript{}{a}{\mathcal{I}}_{x}^{1-\mu} \widetilde{u} \Big|  &=& \lim_{x\rightarrow a} \frac{1}{\Gamma(1-\mu)} \Big| \int_{a}^{x}  \frac{\widetilde{u}(s) ds}{(x-s)^{\mu}} \Big|, 
		\\ \nonumber
		&\leq & 
		\lim_{x\rightarrow a} \frac{1}{\Gamma(1-\mu)} \Big| \int_{a}^{x}  \frac{ ds}{(x-s)^{\mu}} \Big| \Vert \widetilde{u} \Vert_{L^{\infty}},
		\\ \nonumber
		&= & 
		\lim_{x\rightarrow a} \frac{1}{\Gamma(1-\mu)}
		\frac{(x-a)^{1-\mu}}{1-\mu}
		\Vert \widetilde{u} \Vert_{L^{\infty}} =0.
	\end{eqnarray}
	Moreover, it is easy to check that
	\begin{eqnarray}
		\nonumber
		\frac{d}{ds} \int_{s}^{b} \frac{v(x) dx}{(x-s)^{\mu}}
		&=& 
		\frac{d}{ds} \Big\{ 
		\frac{v(x)(x-s)^{1-\mu}}{1-\mu}\Big|_{x=s}^{x=b}
		- 
		\frac{1}{1-\mu}  \int_{s}^{b} \frac{dv(x)}{dx}(x-s)^{1-\mu} dx
		\Big\},
		\\ \nonumber
		&=& 
		\frac{d}{ds} \Big\{ 0 - 
		\frac{1}{1-\mu}  \int_{s}^{b} \frac{dv(x)}{dx}(x-s)^{1-\mu} dx
		\Big\},
		\\ 
		\label{Eq: derivation integ-by-parts2}
		&=& 
		\int_{s}^{b}  \frac{\frac{dv(x)}{dx} dx}{(x-s)^{\mu}}.
	\end{eqnarray}
	Now, by substituting \eqref{Eq: derivation integ-by-parts2} into \eqref{Eq: derivation integ-by-parts1}, we obtain 
	\begin{eqnarray}
		\nonumber
		( \prescript{}{a}{\mathcal{D}}_{x}^{2\mu} u \,,\,v ) &=&
		\int_{a}^{b}  
		\widetilde{u}(s) 
		\Big\{
		\frac{1}{\Gamma(1-\mu)} 
		(\frac{-d}{ds}) \int_{s}^{b} \frac{v(x) dx}{(x-s)^{\mu}}
		\Big\}
		ds,
		\\ \nonumber
		&=&
		(  \widetilde{u} \,,\,\prescript{}{x}{\mathcal{D}}_{b}^{\mu}v ),
	\end{eqnarray}
	when $\prescript{}{x}{\mathcal{D}}_{b}^{\mu}v$ is well-defined and is integrable in the interval $[a,b]$. 
\end{proof}
%

%
\section{Proof of Theorem \eqref{Thm: quadrature convergence}}
\label{Sec: App. proof quadrature convergence}
%
\begin{proof}

	\noindent \underline{Part A}:
	
	The Jacobi polynomials, $P^{\alpha , \beta}_{n}(\xi)$, can be constructed via the three-term recursion relation.
	By letting $\beta = -\alpha$,
	the corresponding tree term recursion reduces to 
	\begin{equation}
		\label{Eq: Three term recursoin reduced}
		P^{\alpha , -\alpha}_{n+1}(\xi)  = \frac{(2n + 1)}{(n + 1)} \, \xi \,\, P^{\alpha , -\alpha}_{n}(\xi) - \frac{(n - \alpha^2)}{n (n + 1)} \, \xi \,\, P^{\alpha , -\alpha}_{n-1}(\xi), 
	\end{equation}
	and therefore, the Jacobi polynomials evaluated at $\xi = \xi_0 \in [-1 , 1]$ are obtained in the following standard form
	\begin{align}
		P^{\alpha , -\alpha}_{0}(\xi_0) & = \varmathbb{P}_0(\alpha) = 1, & & : \text{zeroth order in } \alpha&
		\\
		\nonumber
		P^{\alpha , -\alpha}_{1}(\xi_0) & = \varmathbb{P}_1(\alpha) = \alpha + \xi_0, & & : \text{linear in } \alpha&
		\\
		\nonumber
		P^{\alpha , -\alpha}_{2}(\xi_0) & = \varmathbb{P}_2(\alpha) = \frac{1}{2} \, \alpha^2 + \frac{3}{2} \, \xi_0 \, \alpha + \frac{3 \, \xi_0^2 -1}{2},  & & : \text{quadratic  in } \alpha&
		\\
		\nonumber
		P^{\alpha , -\alpha}_{3}(\xi_0) & = \varmathbb{P}_3(\alpha) = \frac{1}{6} \, \alpha^3 + \xi_0 \, \alpha^2 + \frac{15 \, \xi_0^2 - 4}{6} \, \alpha + \frac{5 \, \xi_0^3 - 3 \, \xi_0}{2}.  & & : \text{cubic  in } \alpha&
	\end{align}
	Now, let $n = k$, thus, $P^{\alpha , -\alpha}_{k}(\xi_0)$ and $P^{\alpha , -\alpha}_{k-1}(\xi_0)$ are respectively polynomials of order $k$ and $k-1$ in $\alpha$. Using  \eqref{Eq: Three term recursoin reduced} for $n = k + 1$, we get   
	\begin{equation*}
		P^{\alpha , -\alpha}_{k+1}(\xi_0) = \frac{(2k + 1)}{(k + 1)} \, \xi_0 \,\, P^{\alpha , -\alpha}_{k}(\xi_0) - \frac{(k - \alpha^2)}{k (k + 1)} \, \xi_0 \,\, P^{\alpha , -\alpha}_{k-1}(\xi_0),
	\end{equation*}
	which is a polynomials of order $k+1$ in $\alpha$ due to the second term. Hence, by mathematical induction, $P^{\alpha , -\alpha}_{n}(\xi_0) = \varmathbb{P}_n(\alpha)$ is a polynomial of order $n$ in $\alpha$ $\forall \xi_0 \in [-1 , 1]$. Similarly with the same argument, we can show that $P^{-\alpha , \alpha}_{n}(\xi_0) = \varmathbb{P}_n(-\alpha)$ is also a polynomial of order $n$ in $\alpha$ $\forall \xi_0 \in [-1 , 1]$.

	\vspace{0.2 in}
	\noindent \underline{Part B}:
	
	The inner integral of the discrete distributed bilinear form \eqref{Eq: discrete distributed bilinear form UN rep 3} can be written as
	\begin{align}
		\label{Eq: Integarl in alpha 1}
		&
		\int_{\alpha_{min}}^{\alpha_{max}} 
		\phi(\alpha) \, (\frac{2}{T})^{\alpha} \,
		\frac{\Gamma(n +\mu)}{\Gamma(n + \eta)}  \frac{\Gamma(k +\mu)}{\Gamma(k + \eta)}  
		\prescript{(1)}{}{ \mathcal{P}}_{n}^{\eta}(\xi) \,\,
		\prescript{(2)}{}{ \mathcal{P}}_{k}^{\eta}(\xi) \,\,
		\, d\alpha = 
		\\
		\nonumber
		&
		\Gamma(n +\mu) \Gamma(k +\mu)
		\int_{\alpha_{min}}^{\alpha_{max}} 
		\phi(\alpha) \,\, (\frac{2}{T})^{\alpha} \,
		\frac{(1+\xi)^{\eta}}{\Gamma(n + \eta)}  \frac{(1-\xi)^{\eta}}{\Gamma(k + \eta)}  
		\,\,
		P^{-\eta , \eta}_{n-1}(\xi) \,\,
		P^{\eta , -\eta}_{k-1}(\xi) \,\,
		d\alpha,
	\end{align}
	in which $\eta = \mu - \alpha/2$. By theorem \eqref{Thm: quadrature convergence} part A, $P^{\frac{\alpha}{2}-\mu , \mu-\frac{\alpha}{2}}_{n-1}(\xi)$ and $P^{\, \mu - \frac{\alpha}{2} , \frac{\alpha}{2} - \mu}_{k-1}(\xi)$ are polynomials in $\alpha$ of order $n-1$ and $k-1$, respectively, $\forall \xi \in [-1 , 1]$, and $\mu$ fixed. Thus,
	\begin{align}
		\label{Eq: Jacobi in terms of alpha n}
		P^{\frac{\alpha}{2}-\mu , \mu-\frac{\alpha}{2}}_{n-1}(\xi) \,\,
		&=
		\sum_{r=0}^{n-1} \sigma_r \, P_r(\alpha) ,
		\\
		\label{Eq: Jacobi in terms of alpha k}
		P^{\, \mu-\frac{\alpha}{2} , \frac{\alpha}{2}-\mu}_{k-1}(\xi) \,\,
		&=
		\sum_{l=0}^{k-1} \widetilde{\sigma}_l\, P_l(\alpha).
	\end{align}
	By plugging \eqref{Eq: Jacobi in terms of alpha n} and \eqref{Eq: Jacobi in terms of alpha k} into \eqref{Eq: Integarl in alpha 1}, we obtain
	\begin{align}
		\nonumber
		&
		\int_{\alpha_{min}}^{\alpha_{max}}  \phi(\alpha)  \,\, (\frac{2}{T})^{\alpha} \,
		\frac{(1+\xi)^{\eta}}{\Gamma(n + \eta)}  \frac{(1-\xi)^{\eta}}{\Gamma(k + \eta)}  
		\,\,
		P^{-\eta , \eta}_{n-1}(\xi) \,\,
		P^{\eta , -\eta}_{k-1}(\xi) \,\,
		d\alpha =
		\\
		\label{Eq: Integarl in alpha 2}
		&
		\int_{\alpha_{min}}^{\alpha_{max}}  
		\phi(\alpha)  \,\,  \mathcal{W}_{kn}^{\,\xi,\mu}(\alpha)
		\,\,
		\sum_{r=0}^{n-1} \sigma_r \, P_r(\alpha) \,\,
		\sum_{l=0}^{k-1} \widetilde{\sigma}_l\, P_l(\alpha) \,\,
		d\alpha,
	\end{align}
	in which
	\begin{equation}
		\nonumber
		\mathcal{W}_{kn}^{\, \xi,\mu}(\alpha) = 
		(\frac{2}{T})^{\alpha} \,\,
		\frac{(1-\xi)^{\mu - \alpha}}{\Gamma(n + \mu - \alpha)}  \frac{(1+\xi)^{\mu - \alpha}}{\Gamma(k + \mu - \alpha)}
	\end{equation}
	is \textit{smooth} in any compact support in $[\alpha_{min} , \alpha_{max}]$ and its polynomial expansion $\mathcal{W}_{kn}^{\, \xi,\mu}\big|_{N} (\alpha) = \mathcal{W}_{N}(\alpha) = \sum_{q=0}^{N} \,\, \rho_q \, P_q(\alpha)$ converges exponentially i.e., 
	\begin{equation}
		\label{Eq: App. convergence of smooth function}
		\norm{\mathcal{W}_{kn}^{\, \xi,\mu}(\alpha) - \mathcal{W}_{N}(\alpha)} \leq c_1 \exp(-c_2 N^{c_3}),
	\end{equation}
	in which $\norm{.}$ denotes the $L^2$-norm in $[\alpha_{min} , \alpha_{max}]$. If the distribution function $\phi \in H^r([\alpha_{min} , \alpha_{max}])$, $r>0$, we have the following projection error:
	\begin{equation}
		\label{Eq: App. convergence of weight function}
		\norm{\phi(\alpha) - \phi_N (\alpha)} \leq c_4 \, N^{-r} \norm{\phi}_{H^r([\alpha_{min} , \alpha_{max}])},
	\end{equation}
	where $\phi_N (\alpha) = \sum_{n=0}^{N} \,\, \widetilde{\rho}_n \, P_n(\alpha)$. Consequently, the integrand in \eqref{Eq: Integarl in alpha 2} can be well-approximated via 
	\begin{align}
		\label{Eq: Integarl in alpha integrand}
		\phi(\alpha)  \,\,  \mathcal{W}_{kn}^{\, \xi,\mu}(\alpha) \,\,
		\sum_{r=0}^{n-1} \sigma_r \, P_r(\alpha) \,\,
		\sum_{l=0}^{k-1} \widetilde{\sigma}_l\, P_l(\alpha) \,\,
		\approx
		\phi_N(\alpha)  \,\,  \mathcal{W}_N(\alpha) \,\,
		\sum_{r=0}^{n-1} \sigma_r \, P_r(\alpha) \,\,
		\sum_{l=0}^{k-1} \widetilde{\sigma}_l\, P_l(\alpha).
	\end{align}
	Next, let
	\begin{align}
		\label{Eq: Integarl in alpha exact}
		\mathcal{I} 
		& = 
		\int_{\alpha_{min}}^{\alpha_{max}}  
		\phi(\alpha)  \,\,  \mathcal{W}_{kn}^{\, \xi,\mu}(\alpha)
		\,\,
		\sum_{r=0}^{n-1} \sigma_r \, P_r(\alpha) \,\,
		\sum_{l=0}^{k-1} \widetilde{\sigma}_l\, P_l(\alpha) \,\,
		d\alpha,
		\\
		\nonumber
		\mathcal{I}_N 
		& =
		\int_{\alpha_{min}}^{\alpha_{max}}  
		\phi_N(\alpha)  \,\,  \mathcal{W}_N(\alpha)
		\,\,
		\sum_{r=0}^{n-1} \sigma_r \, P_r(\alpha) \,\,
		\sum_{l=0}^{k-1} \widetilde{\sigma}_l\, P_l(\alpha) \,\,
		d\alpha,
	\end{align}
	where $\mathcal{I}_N$ can be \textit{accurately} calculated via
	\begin{align}
		\label{Eq: Integarl in alpha approximate}
		\mathcal{I}_N 
		& =
		\sum_{q = 1}^{Q} 
		\, \widetilde{w}_q \, \phi_N(\alpha_q) \, \mathcal{W}_N(\alpha_q)  \,\,
		\sum_{r=0}^{n-1} \sigma_r \, P_r(\alpha_q) \,\,
		\sum_{l=0}^{k-1} \widetilde{\sigma}_l\, P_l(\alpha_q),
	\end{align}
	employing a Gauss-Legendre quadrature rule, provided $Q = 2 N$. Thus by Cauchy-schwarz inequality,
	%
	\begin{align}
		\label{Eq: Integarl in alpha convergency}
		\abs{\mathcal{I} - \mathcal{I}_N}
		\leq 
		\sqrt{\alpha_{min} - \alpha_{max}} \, 
		\norm{
			\phi \, \mathcal{W}_{kn}^{\xi,\mu} \, \sum_{r=0}^{n-1} \sigma_r  P_r \,\, \sum_{l=0}^{k-1} \widetilde{\sigma}_l P_l
			- \phi_N \, \mathcal{W}_{N} \, \sum_{r=0}^{n-1} \sigma_r  P_r \,\, \sum_{l=0}^{k-1} \widetilde{\sigma}_l P_l
		},
	\end{align}
	in which
	\begin{align}
		\label{Eq: App. convergence of smooth function times weight function}
		&\norm{
			\left( \phi \, \mathcal{W}_{kn}^{\,\xi,\mu} \, \sum_{r=0}^{n-1} \sigma_r  P_r \,\, \sum_{l=0}^{k-1} \widetilde{\sigma}_l P_l \right)
			\, - \,
			\left( \phi_N \, \mathcal{W}_{N} \, \sum_{r=0}^{n-1} \sigma_r  P_r \,\, \sum_{l=0}^{k-1} \widetilde{\sigma}_l P_l \right)
		},
		\\
		\nonumber
		\leq \, &
		\norm{\sum_{r=0}^{n-1} \sigma_r  P_r \,\, \sum_{l=0}^{k-1} \widetilde{\sigma}_l P_l}
		\norm{\phi \, \mathcal{W}_{kn}^{\, \xi,\mu} - \phi_N \, \mathcal{W}_{N}}
		,\quad \text{(by H$\ddot{\text{o}}$lder inequality)},
		\\
		\nonumber
		\leq \, &
		c_5
		\norm{\phi \, \mathcal{W}_{kn}^{\, \xi,\mu} - \phi_N \, \mathcal{W}_{N}},
		\\
		\nonumber
		\leq \, &
		c_5
		\norm{\big[(\phi - \phi_N ) + \phi_N \big]
			\big[(\mathcal{W}_{kn}^{\, \xi,\mu} - \mathcal{W}_{N}) + \mathcal{W}_{N}\big] \, 
			- \phi_N \, \mathcal{W}_{N}},
		\\
		\nonumber
		\leq \, &
		c_5
		\norm{
			(\phi - \phi_N ) (\mathcal{W}_{kn}^{\, \xi,\mu} - \mathcal{W}_{N})   
			+ \phi_N (\mathcal{W}_{kn}^{\, \xi,\mu} - \mathcal{W}_{N}) 
			+ (\phi - \phi_N ) \mathcal{W}_{N} 
			+ \phi_N \, \mathcal{W}_{N} 
			- \phi_N \, \mathcal{W}_{N}
		},
		\\
		\nonumber
		\leq \, &
		c_5
		\left(
		\norm{(\phi - \phi_N ) (\mathcal{W}_{kn}^{\, \xi,\mu} - \mathcal{W}_{N})}
		+ \norm{\phi_N (\mathcal{W}_{kn}^{\, \xi,\mu} - \mathcal{W}_{N})}
		+ \norm{(\phi - \phi_N ) \mathcal{W}_{N}}
		\right)
		,\quad \text{(by triangle inequality)},
		\\
		\nonumber
		\leq \, &
		c_5
		\left(
		\norm{(\phi - \phi_N )} \norm{(\mathcal{W}_{kn}^{\, \xi,\mu} - \mathcal{W}_{N})}
		+ \norm{\phi_N} \norm{(\mathcal{W}_{kn}^{\, \xi,\mu} - \mathcal{W}_{N})}
		+ \norm{(\phi - \phi_N )} \norm{\mathcal{W}_{N}}
		\right),
		\\
		\nonumber
		\leq \, &
		c_5 \,
		\left( 
		c_4 \,\, N^{-r} \norm{\phi}_{H^r([\alpha_{min} , \alpha_{max}])} \, . \, c_1 \exp(-c_2 N^{c_3})
		+ \norm{\phi_N} c_1 \exp(-c_2 N^{c_3})
		+ \norm{\mathcal{W}_{N}} c_4 \,\, N^{-r} \norm{\phi}_{H^r([\alpha_{min} , \alpha_{max}])}
		\right),
		\\
		\nonumber
		& \quad\quad \text{(by \eqref{Eq: App. convergence of smooth function} and \eqref{Eq: App. convergence of weight function})},
		\\
		\nonumber
		\leq \, &
		c_6 \, N^{-r} \norm{\phi}_{H^r([\alpha_{min} , \alpha_{max}])}.
	\end{align}
	Hence, by \eqref{Eq: Integarl in alpha convergency} and \eqref{Eq: App. convergence of smooth function times weight function} we can show
	\begin{align}
		\label{Eq: Integarl in alpha convergency 2}
		\abs{\mathcal{I} - \mathcal{I}_N}
		\leq \, 
		\mathscr{C} \, N^{-r} \norm{\phi}_{H^r([\alpha_{min} , \alpha_{max}])},
	\end{align}
	and therefore, by \eqref{Eq: Jacobi in terms of alpha n}, \eqref{Eq: Jacobi in terms of alpha k}, \eqref{Eq: Integarl in alpha exact} and \eqref{Eq: Integarl in alpha approximate}, we obtain
	\begin{align*}
		\bigg|
		&\int_{\alpha_{min}}^{\alpha_{max}}  
		\phi(\alpha) \, (\frac{2}{T})^{\alpha} \,
		\frac{\Gamma(n +\mu)}{\Gamma(n + \eta)}  \frac{\Gamma(k +\mu)}{\Gamma(k + \eta)}  
		\prescript{(1)}{}{ \mathcal{P}}_{n}^{\eta}(\xi) \,\,
		\prescript{(2)}{}{ \mathcal{P}}_{k}^{\eta}(\xi) \,\,
		d\alpha
		\\
		&
		\quad\quad
		-
		\sum_{q=1}^{Q} 
		\, \widetilde{w}_q \, 
		\phi_N(\alpha_q) \, (\frac{2}{T})^{\alpha_q} \,
		\frac{\Gamma(n +\mu)}{\Gamma(n + \eta_q)}  \frac{\Gamma(k +\mu)}{\Gamma(k + \eta_q)}  
		\prescript{(1)}{}{ \mathcal{P}}_{n}^{\eta_q}(\xi) \,\,
		\prescript{(2)}{}{ \mathcal{P}}_{k}^{\eta_q}(\xi) \,\,
		\bigg|
		\\
		&
		\leq
		\mathscr{C} \, Q^{-r} \norm{\phi}_{H^r([\alpha_{min} , \alpha_{max}])}.
	\end{align*}

	\vspace{0.2 in}
	\noindent \underline{Part C}: 
	
	If $\phi(\alpha)$ is smooth, then the approximation $\phi_N(\alpha)$, in \eqref{Eq: App. convergence of weight function}, converges with an exponential accuracy and so does the norm in \eqref{Eq: App. convergence of smooth function times weight function}. Thus,
	\begin{align}
		\label{Eq: Integarl in alpha convergency 3}
		\abs{\mathcal{I} - \mathcal{I}_N}
		\leq \,
		\mathscr{C}_1 \, \text{exp}(- \mathscr{C}_2 \, N^{\mathscr{C}_3}),
	\end{align}
	and therefore, the quadrature rule becomes exponentially accurate in $Q$.

	%

\end{proof}


\newpage
\bibliographystyle{siam}
\bibliography{Distributed_Ref2}

\end{document}